\documentclass{article}
\usepackage{times}
\usepackage[hyperindex=true,pageanchor=true,hyperfigures=true,backref=false]{hyperref} 
\usepackage{amsmath}
\usepackage{amssymb}
\usepackage{amsthm}
\usepackage[capitalise,nameinlink]{cleveref}
\usepackage{mathtools}
\usepackage{tikz-cd}

\usepackage{enumerate}

\usepackage{url}
\usepackage{dsfont} 
\DeclareSymbolFontAlphabet{\Bbb}{AMSb}
\usepackage[T1]{fontenc}

\usepackage{mleftright,xparse}%

\usepackage{thm-restate,amsthm}

\definecolor{darkgreen}{rgb}{0,0.6,0}

\newcommand{\done}[1]{{\color{black}{#1}}}

\newcommand{\secondd}{\color{blue}}
\newcommand{\donee}{\color{black}}

\newlength{\myleftmargin}
\usepackage{amsmath}
\usepackage{amssymb}
\usepackage{amsthm}
\usepackage{color}
\DeclareSymbolFontAlphabet{\Bbb}{AMSb}

\newtheorem{theorem}{Theorem} 
\newtheorem{lemma}[theorem]{Lemma}
\newtheorem{proposition}[theorem]{Proposition}
\newtheorem{corollary}[theorem]{Corollary}
\newtheorem{definition}[theorem]{Definition}

\newtheorem*{question}{Question}
\newtheorem{example}[theorem]{Example}

\newcommand{\atob}[2]{\emph{#1)} $\Rightarrow$ \emph{#2)}.} 
 
\newcommand{\ada}[1]{\emph{#1)\,}} 

\newenvironment{proofof}[1]{\noindent{\bf Proof of #1:}}{\qed\medskip}

\newlength{\fixboxwidth}
\newcommand{\ca}[1]{{\cal #1}}

\newcommand{\qqqquad}{\qquad\qquad}

\newcommand{\mycdot}{\,\cdot\,}

\newcommand{\Z}{\Bbb{Z}}  
\newcommand{\N}{\mathbb{N}}

\newcommand{\R}{\mathbb{R}}
\newcommand{\Rd}{\mathbb{R}^d}

\newcommand{\E}{\mathbb{E}}

\newcommand{\dx}[1]{\hspace*{0.25ex}d\hspace*{-0.15ex}#1}

\renewcommand{\a}{\alpha}
\renewcommand{\b}{\beta}
\newcommand{\g}{\gamma}
\renewcommand{\d}{\delta}

\newcommand{\e}{\varepsilon}

\newcommand{\lb}{\lambda}

\newcommand{\s}{\sigma}

\newcommand{\Om}{\Omega}

\DeclareMathOperator{\diam}{diam}
\DeclareMathOperator{\supp}{supp}

\DeclareMathOperator{\vol}{vol}
\DeclareMathOperator{\ran}{ran}

\DeclareMathOperator{\id}{id}
\newcommand{\eins}{\boldsymbol{1}}

\newcommand{\snorm}[1]{\Vert #1 \Vert}

\newcommand{\bnorm}[1]{\Bigl\Vert \, #1 \, \Bigr\Vert}

\newcommand{\inorm}[1]{\Vert #1 \Vert_\infty}
\newcommand{\tvnorm}[1]{\Vert #1 \Vert_{\mathrm{TV}}}

\newcommand{\sLx}[2]{{\ca L_{#1}(#2)}}

\newcommand{\Lx}[2]{{L_{#1}(#2)}}

\DeclareMathOperator{\rad}{Rad}

\DeclareMathOperator{\hoel}{H\ddot o l}

\DeclareMathOperator{\TV}{TV}

\newcommand{\TVnorm}[1]{\norm{#1}_{\TV}}

\newcommand{\met}{d}
\newcommand{\pac}[3]{N^{\mathrm{pack}}_{#2}(#3,#1)}
\newcommand{\paceuclid}[2]{N^{\mathrm{pack}}(#2, #1)}

\newcommand{\metbump}[3]{f_{#1,#2,#3}}

\DeclareMathOperator{\cotype}{cotype}
\DeclareMathOperator{\type}{type}

\newcommand{\slospace}[3]{W^{#1}_{#2}(#3)}
\newcommand{\preslospace}[2]{W_{\text{aux}}^{#1}(#2)}
\newcommand{\slonorm}[4]{\snorm{#1}_{\slospace{#2}{#3}{#4}}}

\newcommand{\radspace}[2]{\rad^{#1}(#2)}
\newcommand{\radnorm}[3]{\snorm{#1}_{\radspace{#2}{#3}}}

\newcommand{\oldcotypenorm}[3]{\snorm{#1}_{\cotype_2}}

\newcommand{\cotypenorm}[1]{\snorm{#1}_{\cotype_2}}

\newcommand{\oldtypenorm}[3]{\snorm{#1}_{\type_2}}

\newcommand{\typenorm}[1]{\snorm{#1}_{\type_2}}

\newcommand{\precotypenorm}[3]{\snorm{#1}_{\cotype_2}}

\newcommand{\pretypenorm}[3]{\snorm{#1}_{\type_2}}

\newcommand{\ellpspace}[3]{\ell_{#1}^{#2}(#3)}
\newcommand{\ellpnorm}[4]{\snorm{#1}_{\ellpspace{#2}{#3}{#4}}}

\newcommand{\elltwonorm}[3]{\snorm{#1}_{\ellpspace{2}{#2}{#3}}}

\newcommand{\lpspace}[2]{L_{#1}(#2)}
\newcommand{\lpnorm}[3]{\snorm{#1}_{\lpspace{#2}{#3}}}

\newcommand{\holspace}[2]{\hoel^{#1}(#2)}

\newcommand{\twofacnorm}[1]{\snorm{#1}_{\gamma_2}}

\newcommand{\slobodeckij}[4]{\snorm{#1}_{#2,#3,#4}}

\newcommand{\mholnorm}[3]{\snorm{#1}_{\hoel^{#2}_{#3}}}
\newcommand{\Lmholnorm}[3]{\bnorm{#1}_{\hoel^{#2}_{#3}}}
\newcommand{\mholspace}[3]{\hoel^{#2}_{#3}(#1)}

\newcommand{\sobspace}[3]{H^{#1}_{#2}(#3)}

\newcommand{\triebelspace}[4]{F^{#1}_{#2,#3}(#4)}

\newcommand{\besovspace}[4]{B^{#1}_{#2,#3}(#4)}

\newcommand{\supnorm}[1]{\snorm{#1}_\infty}

\newcommand{\munorm}[1]{\abs{#1}_1}

\newcommand{\fun}{f}

\newcommand{\gun}{g}

\newcommand{\pp}{p}
\newcommand{\qq}{q}

\newcommand{\pospart}[1]{(#1)_+\,}
\newcommand{\negpart}[1]{(#1)_-\,}

\newcommand{\naturalsequence}{}

\newcommand{\embe}[2]{ #1 \hookrightarrow #2}

\newcommand{\ball}[2]{B(#2, #1)}

\newcommand{\radsequence}[1]{#1\sim\rad}

\newcommand{\wasoncered}[1]{#1}

\newlength\Origarrayrulewidth

\NewDocumentCommand\xDeclarePairedDelimiter{mmm}
{%
	\NewDocumentCommand#1{som}{%
		\IfNoValueTF{##2}
		{\IfBooleanTF{##1}{#2##3#3}{\mleft#2##3\mright#3}}
		{\mathopen{##2#2}##3\mathclose{##2#3}}%
	}%
}

\newcommand{\floor}[1]{\lfloor #1 \rfloor} 

\newcommand{\negpow}[2]{#1^{-#2}}

\newcommand{\Prob}{\mathbb{P}}

\xDeclarePairedDelimiter{\abs}{\lvert}{\rvert}
\xDeclarePairedDelimiter{\norm}{\lVert}{\rVert}
\xDeclarePairedDelimiter{\set}{\lbrace}{\rbrace}
\newcommand{\setrange}[2]{\lbrace #1,\dots, #2\rbrace}
\xDeclarePairedDelimiter{\bra}{(}{)}
\xDeclarePairedDelimiter{\sbra}{[}{]}
\newcommand{\isum}{\sum_{i=1}^n}
\newcommand{\jsum}{\sum_{j=1}^n}

\newcommand{\scal}[2]{\left\langle #1,#2\right\rangle}			%

\newcommand{\calA}{\mathcal{A}}

\newcommand{\calL}{\mathcal{L}}
\newcommand{\calM}{\mathcal{M}}
\newcommand{\calN}{\mathcal{N}}

\newcommand{\calS}{\mathcal{S}}

\newcommand{\equalDef}{\coloneqq}

\allowdisplaybreaks

\newcommand{\ha}[1]{\hat{#1}}
\newcommand{\ch}[1]{\check{#1}}

\renewcommand{\ss}{s}
\renewcommand{\tt}{t}

\newcommand{\uu}{u}

\renewcommand{\th}{\theta}
\newcommand{\al}{\alpha}

\newcommand{\cone}{\pospart{d/\pone - d/2}}
\newcommand{\ctwo}{\pospart{d/2- d/\ptwo}}
\newcommand{\czero}{\pospart{d/\pp -d/2}}

\newcommand{\A}{A}
\newcommand{\B}{B}

\newcommand{\pone}{p_1}
\newcommand{\ptwo}{p_2}
\newcommand{\qone}{q_1}
\newcommand{\qtwo}{q_2}
\newcommand{\spaceone}[4]{X^{#1}_{#2,#3}(#4)}
\newcommand{\spacetwo}[4]{Y^{#1}_{#2,#3}(#4)}
\newcommand{\Th}{\Theta}

\newcommand{\be}{\beta}

\newcommand{\hfun}{h}
\newcommand{\fnum}[1]{f_1}

\newcommand\restr[2]{{%
		\left.\kern-\nulldelimiterspace %
		#1 %
		\vphantom{\big|} %
		\right|_{#2} %
}}

\newcommand{\dual}[1]{#1'}
\newcommand{\adjoint}[1]{#1^*}

\title{Which Spaces can be Embedded in Reproducing Kernel Hilbert Spaces?}

\author{Max Sch\"olpple and Ingo Steinwart\\
University of Stuttgart\\
\small Faculty 8: Mathematics and Physics\\
\small Institute for Stochastics and Applications\\
\small Stuttgart, Germany \\
\texttt{\small max.schoelpple@mathematik.uni-stuttgart.de}\\
\texttt{\small ingo.steinwart@mathematik.uni-stuttgart.de}
}

\begin{document}

\maketitle

\begin{abstract}
	Given a  Banach space $E$ consisting of functions, we ask whether there exists
	a
	reproducing kernel Hilbert space $H$ with bounded kernel such that $E\subset H$. More generally, we consider the question, whether for a given Banach space consisting of functions $F$ with $E\subset F$,
	there exists an intermediate reproducing kernel Hilbert space $E\subset H\subset F$.
	We provide both sufficient and necessary conditions for this to hold. Moreover, we show that for typical classes of function spaces described by smoothness there is a strong dependence on the underlying dimension:  the smoothness $s$ required for the space $E$ needs  to grow \emph{proportional} to the dimension $d$ in order to allow for an intermediate reproducing kernel Hilbert space $H$.
\end{abstract}

\textbf{Mathematical Subject Classification (2010).} Primary 46E22; Secondary 46B20, 46E15, 46E35, 47A68, 68T05.

\textbf{Key Words.} Reproducing kernel Hilbert spaces, Inclusion maps, Spaces of functions,  Factorization of Operators, Type and Cotype, Besov- and Triebel-Lizorkin spaces, Sobolev spaces with mixed smoothness, Neural networks.

\section{Introduction}\label{sec:intro}

\donee
Reproducing kernel Hilbert spaces (RKHSs) are a powerful concept in various branches of mathematics and applications such as  
numerical analysis \cite{Wendland05,DiKrPi22},
stochastics \cite{BeTA04},
signal processing \cite{Gualtierotti15}
analysis \cite{Saitoh88,AgMc02,Jorgensen18,ArRosaWi19},
and machine learning \cite{ScSm02,CuZh07,StCh08,Pereverzyev22}.  
In many of these applications, understanding RKHSs is central task and the choice of suitable RKHSs is essential in numerous algorithms. 
\donee

Given an RKHS $H$ on some underlying space $X$, properties of its kernel $k$ or its elements $f\in H$ can often be described 
in terms of a surrounding space $F$, that is, by an inclusion $H\subset F$. For example,
it is well-known that 
the kernel $k$ is bounded, 
if and only if $H\subset \ell_\infty(X)$ holds, where $\ell_\infty(X)$ denotes the Banach space of all bounded functions $f:X\to \R$
equipped with the usual supremum norm.
Bounded kernels do not only play a central role in the statistical analysis of kernel-based learning algorithms, 
e.g.~\cite{StCh08},
but they are also key for kernel mean embeddings, see e.g.~\cite{SrGrFuScLa10a,StZi21a}.
Moreover,  the kernel is bounded and continuous, 
if and only if $H\subset C_b(X)$, where $C_b(X)$ denotes
the space of all bounded, continuous functions $f:X\to \R$, 
and generalizations of such inclusions to higher notion of smoothness are, of course, possible. 

In machine learning applications one is usually interested in sufficiently large RKHSs.
Here, one way to describe the ``size'' of an RKHS $H$ is by means of denseness in a surrounding space $F$. For example,
a continuous  kernel on some compact metric space $X$  is said to be universal \cite{Steinwart01a}, 
if $H\subset C(X)$ is dense, where $C(X)$ denotes the 
space of all continuous functions $f:X\to \R$. 
Several papers have investigated universal kernels, see e.g.~\cite{MiXuZh06a,SrFuLa11a,SGSc18a,StZi21a}
and the various references mentioned therein. In addition, recall that basic learning guarantees, such as universal consistency,
can be established for kernel-based learning algorithms  
if universal kernels are used, see e.g.~\cite{StCh08}. Moreover, universal kernels are very closely related to 
characteristic kernels, which play a central role for identifying probability distributions \cite{SrGrFuScLa10a, SrFuLa11a}, e.g.~with the help of 
kernel-based non-parametric two-sample tests \cite{GrBoRaScSm07a,GrBoRaScSm12a}.

Of course, the size of an RKHS $H$ could also be described by
specifying a ``lower bound'' on $H$, that is, a set, or vector space,
of functions $E$ it is supposed to contain.
This is   the focus of this paper. Namely,
we consider the question:

\begin{question}
Given two Banach spaces $E\subset F$ of functions $X\to \R$, does there exist an RKHS $H$ on 
$X$ with 
\begin{align}
	\label{eq:central_question}
	E\subset H \subset F\, ?
\end{align}
\end{question}

Here, the Banach space of functions (BSF) $F$ encodes additional properties an RKHS $H$ with $E\subset H$ is supposed to satisfy. If no additional properties
are needed, negative answers can thus be formulated by saying that \eqref{eq:central_question} is impossible for all BSFs $F$ with $E\subset F$.
For example, if $X$ is an uncountable, compact metric space such as $X=[0,1]^d$, then  \cite{Steinwart24a}
has recently shown that for the BSF $E:= C(X)$ Question  \eqref{eq:central_question}  has a negative answer for all BSFs $F$.
Besides  \cite{Steinwart24a}, however, Question  \eqref{eq:central_question}  has, to the best of our knowledge, not been considered in the literature, yet.
Nonetheless, having positive or negative answers to this question may have various applications as we will discuss later in this introduction.

Let us now briefly summarize the results we obtain in this paper. Here, it seems fair to mention that  
for the greatest part of this paper we investigate Question \eqref{eq:central_question} under the additional assumption that the
point evaluations on $E$ and $F$ are continuous. This additional assumption is, however, at best a mild restriction as it is  
satisfied for practically all BSFs, and by definition,
for all RKHSs. Moreover,
it turns out that under this additional assumption the inclusion maps $\id:E\to H$ and $\id:H\to F$
are automatically continuous, see Lemma \ref{lem:inclusion-is-continuous},
which in turn opens the door for tools from functional
analysis.
Accordingly,
we can freely switch between \eqref{eq:central_question} and $E\hookrightarrow H \hookrightarrow F$, and in particular we will do so in this 
introduction. %

Now,  in the abstract setting of generic BSFs $E$ and $F$, we show,  among other results:
\donee
\begin{itemize}
\item Question \eqref{eq:central_question} has a positive answer if and only if the inclusion map  $\id:E\to F$ is 2-factorable, that is
there exist an (abstract) Hilbert space $H$ and bounded linear operators $U:E\to H$ and $V:H\to V$
with $\id = V\circ U$, such that we have
\begin{center}
		\begin{tikzcd}
		E \arrow[rd,"U"] \arrow[hookrightarrow,rr, "\id"] \pgfmatrixnextcell \pgfmatrixnextcell F \\
		\pgfmatrixnextcell H \arrow[ru,"V"] \pgfmatrixnextcell
	\end{tikzcd},
\end{center}
see \cref{theorem:2-fac_in_bfs}.
Of course, the inclusion \eqref{eq:central_question} gives such a factorization,  but the converse requires an explicit construction of a suitable RKHS.\donee

\item Question \eqref{eq:central_question} has a positive answer if $\id:E\to F$ is 2-summing, see e.g.~\cite[Ch.~2]{DiJaTo95} for a definition 
and key properties,
or if $E$ is of type 2 and $F$ is of cotype 2, see \cref{thm:suff-2-fact}.  
Here we note that these sufficient conditions can be derived from general results on 2-factorable operators, 
see e.g.~\cite{Pisier86}
and  \cite{DiJaTo95}. 

\item If Question \eqref{eq:central_question} has a positive answer, then $\id:E\to F$ has both type 2 and cotype 2, see \cref{cor:nec-2-fact}.
\end{itemize}
For concrete families of spaces, we provide negative and often also positive results in the following cases:
\begin{itemize}
\item
For the spaces $E=\mholspace{X}{\al}{\met}$ of bounded, $\a$-H\"older continuous functions on some metric space, where $\a\in (0,1]$,
we show in \cref{thm:abstract_holder_twofac_thm} that
a positive answer to Question \eqref{eq:central_question} with $F=\ell_\infty(X)$
implies a strong upper bound on the covering, or packing
numbers of $X$.
As a consequence, for the open, $k$-dimensional cube $X:=(0,1)^k$ and $\a<k/2$ no positive answer is possible,  
while for $\a \in (k/2,1)$ we have a positive answer, see \cref{lemma:besov_c0_twofac}. In particular, for $k\geq 3$,   none of the spaces  
$\holspace{\al}{(0,1)^k}{}$ %
can be embedded into an RKHS, while for $k=1$ such an embedding is possible for all $\a\in (1/2,1)$ using a fractional Sobolev space $H=\sobspace{\uu}{2}{(0,1)}$ for $u\in (1/2,\a)$.
Moreover, in this case the inclusion $\mholspace{(0,1)}{1}{} \subset \mholspace{(0,1)}{1-\varepsilon}{}$, which holds for all $0<\varepsilon<1/2$,  
shows that 
$\mholspace{(0,1)}{1}{}$ can also be embedded into the RKHS $\sobspace{\uu}{2}{(0,1)}$ with bounded kernel for any $\uu\in (1/2,1)$.

\item
For Sobolev spaces $E$ and $F$ of integer (mixed) smoothness on some bounded domain $X\subset \Rd$ we show that if  Question \eqref{eq:central_question} has a positive answer, then the difference of the involved smoothness parameters needs to be
sufficiently large, see \cref{thm:mixed_twofac} and \cref{thm:mixed_twofac_C-0}. In the case of Sobolev spaces this difference needs to grow linearly with the dimension $d$, while for
Sobolev spaces with mixed smoothness, this is not the case. \donee
\item For Sobolev-Slobodeckij spaces $E$ and $F$ on some bounded domain $X\subset \Rd$  we  provide positive and negative results
that match to each other modulo some limit cases. In a nutshell, positive results are possible if and only if the difference of the involved smoothness parameters is sufficiently large compared to $d$, see \cref{corollary:parameter_reqs_in_slob_spaces} for details.
The same is true if we consider $F=\ell_\infty(X)$, instead.
\item The results for  Sobolev-Slobodeckij spaces can be extended to spaces  $E$ and $F$ from the Besov or Triebel-Lizorkin family of spaces,
see \cref{thm:two_factorisability_of_general_triebel_besov_embedds} and \cref{lemma:besov_c0_twofac}.
\end{itemize}
Here we note that   our   negative answers to Question  \eqref{eq:central_question}  are always derived by showing that
the inclusion map $\id:E\to F$ fails to have type 2 or cotype 2. On the other hand, we always construct positive results by
employing well-known embedding theorems for the involved spaces.
 Moreover, in some cases our positive and negative results complement each other (modulo some limit cases), and therefore
our type 2/cotype 2 technique is at least sometimes sharp. In these cases, we further see (modulo limit cases)
that \eqref{eq:central_question}  has a positive answer, if and only if one already
  knows   an RKHS 
  $H$   from the literature.

Let us finally illustrate the potential impact of positive and negative answers to Question  \eqref{eq:central_question}.
Here, our first example considers kernel-based learning algorithms, see e.g.~\cite{CuZh07,StCh08}.
For such learning algorithms,
standard (and fixed)  RKHSs
are sometimes viewed to be too small.
To be specific, while the Gaussian kernel with width $\g$ is known to be universal for every single value of $\g$, 
it is also known that from a quantitative perspective it only slowly approximates non-$C^\infty$-functions, see e.g.~\cite{SmZh03a}.
For this reason one either lets $\g$ depend on the sample (size), see e.g.~\cite{EbSt13a}, or considers e.g.~so-called hyper-RKHS, see \cite{OnSmWi05a,LiShHuYaSu21a}.
Rather than employing such approaches, one could, however, also look for e.g.~tailored RKHSs:
Namely, if we have a space $E$ of target functions
for which we like to construct a kernel-based algorithm with ``good'' learning behaviour,  an RKHS $H$ with $E\subset H\subset F$ can be desirable.
Indeed, $E\subset H$ ensures a good approximation error during the analysis of the resulting algorithm, while $H\subset F$ can provide a strong
 eigenvalue or entropy number decay, see e.g.~\cite[Lem.~7.21 and Thm.~6.26]{StCh08} for the cases $H\subset \ell_\infty(X)$ or $H\subset C^\a(X)$,
 which in turn leads to a small estimation error.

Here we note that recently, BSFs 
have attracted interest in the machine learning  community since neural networks define particular forms of such spaces, see
 e.g.~\cite{BaDVRoVi23a}
and the references mentioned therein. 
In addition, \cite[Sec.~2.3]{Bach17a} compares the differences between neural network specific BSFs $E$
and related (but unfortunately smaller) RKHSs $H$ in some detail. 
 We provide sufficient requirements for the corresponding integral reproducing kernel Banach spaces $E$ to be included in an RKHS $H$ in \cref{subsubsection:RKBS}.

Another possible application are so-called integral probability metrics, see e.g.~\cite{Mueller97a} and the references mentioned therein.
Given a measurable space $(X,\ca A)$ and a set $\ca F$ of  measurable functions $X\to \R$, these metrics are defined as 
\begin{align*}
\g_{\ca F}(P,Q) \equalDef \sup_{f\in \ca F} \Bigl| \int_X f \, dP - \int_X f \,dQ \Bigr|
\end{align*}
for all probability measures $P$ and $Q$ on $(X,\ca A)$ for which the integrals do  exist.
For example, if $\ca F =B_{\sLx \infty X}$
is the unit ball in the space $\sLx \infty X$ of all
bounded measurable functions $X\to \R$ equipped with the supremum norm, then $\g_{\ca F}$ equals the metric obtained from the 
total variation norm.
In addition, if $(X,d)$ is a separable metric space equipped with the Borel $\s$-algebra and
$\ca F$  is the unit ball in the space $\mholspace{X}{1}{\met}$ of all \emph{bounded} Lipschitz functions, then $\g_{\ca F}$ is known as the \emph{Dudley metric}, which metrizes 
the weak convergence of probability measures, see e.g.~\cite[Theorem 11.3.3]{Dudley02}.
Furthermore, if $\ca F$ is the set of all Lipschitz continuous functions with Lipschitz constant $\leq 1$ and $(X,d)$ 
is a bounded metric space, 
then the famous Kantorovich–Rubinstein theorem shows that $\g_{\ca F}$ equals the Wasserstein 1-distance, see e.g.~\cite[Theorem 11.8.2]{Dudley02}.
Finally, if $\ca F$ is the unit ball $B_H$ of an RKHS $H$ with bounded, measurable kernel, then $\g_{\ca F}$ is called
$H$-maximum mean discrepancy (MMD), see \cite{GrBoRaScSm07a}.  Moreover, it can be shown that having a bounded and measurable
kernel is also necessary for the MMD to be defined for all probability measures on $(X,\ca A)$, see e.g.~\cite[Prop.~2]{SrGrFuScLa10a}
in combination with \cite[Lem.~4.23]{StCh08}.
Finally, recall that, unlike many other integral probability metrics, MMDs can be
 both expressed in closed form and estimated from data, see e.g.~\cite{GrBoRaScSm12a}.

This raises the question of how powerful MMDs are compared to the general class of integral probability metrics.
To discuss this question, let us fix an RKHS $H$ with bounded measurable kernel.
Then $\g_{B_H}(P,Q)$ exists for all probability measures $P$ and $Q$ on $X$ and we have 
\begin{align*}
  \g_{B_H}(P,Q)   \leq \snorm{\id: H\to \sLx \infty X}   \cdot  \tvnorm{P-Q}\, .
\end{align*}
Consequently, if an empirical estimate ensures that $\g_{B_H}(P,Q)$ is ``large'' with high probability, then $P$ and $Q$ are 
``rather distinct'' in the sense of the total variation norm.  
However, if $\g_{B_H}(P,Q)>0$ is ``small'', then we cannot guarantee that $\norm{P-Q}_{\mathrm{TV}}$ is sufficiently small
by the inequality above, see also \cite{StZi21a} for a rigorous negative result in this direction. 
To overcome this issue, the current MMD literature focuses on \emph{characteristic} kernels, that is, on kernels for which 
 $\g_{B_H}(P,Q)=0$ implies $P=Q$. To obtain a more quantitative result, assume that we have a BSF $E$ with unit ball $B_E$ and $E\hookrightarrow H$.
 Then the inequality above can be extended to 
\begin{align*}
\snorm{\id: E\to H}^{-1}  \cdot  \g_{B_E}(P,Q) \leq  \g_{B_H}(P,Q)  \leq \snorm{\id: H\to \sLx \infty X}   \cdot  \tvnorm{P-Q} \, .
\end{align*}
Consequently, $\g_{B_H}(P,Q)>0$ is ``small'', then $P$ and $Q$ are ``similar'' in the sense of the integral probability metric 
$\g_{B_E}$.
Depending on the space $E$, the integral probability metric $\g_{B_E}$ might have a clear and intuitive meaning, but it might be difficult or even infeasible to estimate $\g_{B_E}$ from data. 
In this case, MMDs may still be helpful.
Indeed, if we have RKHSs $H_1$ and $H_2$ with $H_1\hookrightarrow E\hookrightarrow H_2\hookrightarrow\sLx \infty X$, then
\begin{align*}
\snorm{\id: H_1\to E}^{-1}  \cdot  \g_{B_{H_1}}(P,Q) \leq  \g_{B_E}(P,Q)   \leq \snorm{\id: E\to H_2}   \cdot  \g_{B_{H_2}}(P,Q) 
\end{align*}
holds for all probability measures $P$ and $Q$ on $X$. Consequently, if we know from data that $\g_{B_{H_1}}(P,Q)$  is large with high probability, 
then, with the same probability,  $P$ and $Q$ are rather distinct 
in the sense of $\g_{B_E}$. Conversely, if we know from data that $\g_{B_{H_2}}(P,Q)$ is small with high probability, then 
with the same probability,  $P$ and $Q$ are rather similar
in the sense of $\g_{B_E}$. In other words, using $H_1$ and $H_2$ we may gain a two-sided \emph{quantitative} control over $\g_{B_E}$.

\donee

In view of this discussion let us now  
 illustrate the consequences of our findings for MMDs. To this end, let $X:=(0,1)^k$.  
 If $k=1$, 
we have already noted above that we have the inclusions
\begin{align*}
\mholspace{(0,1)}{1}{} \subset \sobspace{\uu}{2}{(0,1)} \subset \sLx \infty {(0,1)}
\end{align*}
for all $u\in (1/2,1)$. Consequently, we have MMDs that provide upper bounds for the Dudley metric,  while  RKHS  $H$ with sufficiently smooth kernel
satisfy $H\subset \mholspace{(0,1)}{1}{}$, that is, they provide lower bounds for the Dudley metric. 
On the other hand, in the case $k\geq 3$, the well-established inclusions between Besov spaces do not provide 
a fractional Sobolev space between $\mholspace{(0,1)^k}{1}{}$ and $ \sLx \infty {(0,1)^k}$, and our results now show that 
there is actually no RKHS $H$ with $\mholspace{(0,1)^k}{1}{} \subset H \subset \sLx \infty {(0,1)^k}$.
As a consequence, we do not obtain an MMD-based upper bound of the Dudley metric as soon as the dimension 
satisfies $k\geq 3$, while a lower bound is still possible by using sufficiently smooth kernels.

The examples on learning algorithms and integral probability metrics can be generalized. Namely,  
 if we wish to design an algorithm that deals with functions $f\in E$ via their point evaluations, then
having an RKHS $H$ with $E\subset H$ could have significant algorithmic advantages: Indeed,
the ``kernel trick'' \cite{ScSm02} can simplify the computation of inner products in $H$, which in turn
makes it possible to use inner products in the algorithm design by 
 interpreting a function $f\in E$  as an element of $H$.
Similarly, if  
no such $H$ exists, then  we know that no such algorithmic short-cut is possible.

\donee

The rest of this paper is organized as follows:
In Section \ref{sec:prelim} we recall some notion and  technical tools used in this paper. In addition, most of the notation is fixed.
Section \ref{sec:general} presents all our results in the abstract setting and Section \ref{sec:examples}
contains all results related to concrete spaces. In addition, we discuss our general approach for deriving negative results
for two simple families of spaces.
Finally, all proofs  
can be found in Section \ref{sec:proofs}.

\section{Preliminaries}\label{sec:prelim}

Given a number $\ss\in\R$, we write $\floor\ss\equalDef \max\{k\in\Z\mid k\le \ss\}$ and $\pospart \ss\equalDef \max\{0,\ss\}, \negpart \ss\equalDef \max\{0,-\ss\}$.

In the following, $E$ and $F$ denote Banach spaces with norms $\snorm\cdot_E$ and $\snorm\cdot_F$ and closed
unit balls $B_E$ and $B_F$, respectively. Moreover, $E'$ and $F'$ denote their dual spaces.
In the case $E\subset F$, we write $E\hookrightarrow F$, if the corresponding inclusion map is continuous.

In this work, we are mostly interested in Banach spaces consisting of functions. The following definition
introduces these spaces formally.

\begin{definition}\label{def:bsf}
 Let $X\neq \emptyset$. Then a Banach space $E$ is called a Banach space of functions (BSF) on $X$, if
 all its elements are functions mapping from $X$ to $\R$. Moreover, we say that $E$ is a proper BSF, if
 for all $x\in X$, the evaluation functional
 \begin{align*}
  \d_x:E&\to \R \\
 f &\mapsto f(x)
 \end{align*}
 is continuous. A proper BSF with a Hilbert space norm is called reproducing kernel Hilbert space (RKHS).
\end{definition}

In the literature, Banach spaces of functions are also called \emph{reproducing kernel Banach spaces}. These spaces
have been recently gained interest for the analysis of neural networks, see e.g.~\cite{BaDVRoVi23a}.

The usual sequence spaces $c_0$  and $\ell_p$ for $p\in [1,\infty]$ are all proper BSFs.
Moreover, the space of $\R$-valued, bounded continuous functions $C^0(X)$ on some metric space $X$
is also a proper BSF, while the usual Lebesgue spaces on e.g.~the unit interval, that is, $\Lx p {[0,1]}$,
fail to be a BSF. Finally, the space $\ell_\infty(X)$ of all bounded functions $f:X\to \R$ equipped with
the supremum norm   is also a proper BSF.

Note that in all proper BSFs $E$ on $X$ norm-convergence implies pointwise convergence, that is, if $(f_n)\subset E$
is a sequence converging to some $f\in E$ in the sense of $\snorm{f_n-f}_E\to 0$, then
$f_n(x)\to f(x)$ for all $x\in X$.

RKHSs have been studied in great detail in the literature. For basic information we refer to e.g.~\cite[Chapter 4]{StCh08}. 
For reproducing kernel Hilbert spaces $H$ on $X$, the corresponding (reproducing) kernel $k:X\times X\to \R$
can be defined by
$$
k(x,x') \equalDef \langle \d_x,\d_{x'}\rangle_{H'}\, , \qqqquad x,x'\in X.
$$
It is well-known that various properties of the functions in $H$ can be characterized by properties of
$k$. For example, all functions $f\in H$ are bounded if and only if the kernel $k$ is bounded,
see e.g.~\cite[Lem.~4.23]{StCh08}. Consequently,
if we have a proper BSF $E$ and we seek a surrounding RKHS $H$ with bounded kernel, we can express this
equivalently as
\begin{align}\label{eq:linfty-inclusion}
E\subset H\subset \ell_\infty(X)\, .
\end{align}
Note that in this case, the inclusion maps are automatically continuous, see Lemma \ref{lem:inclusion-is-continuous}.

\donee

Many results of this paper rely on the type and cotype of operators and spaces. For this reason, let
us quickly recall these notions.
To this end, we fix a Banach space $E$ and some $n\geq 1$.
Then we define a norm  on the $n$-fold product space $E^n$ by
\begin{align*}
    \ellpnorm{(x_1,\dots,x_n)}2nE &\equalDef
    \bra[\Big]{ \isum \norm{x_i}_E^2}^{1/2}\, , \qquad \qquad x_1,\dots,x_n\in E,
\end{align*}
which we sometimes call the \emph{sequence norm}.
Moreover, let $\e =(\e_i)$  
be a
\emph{Rademacher sequence}, that is, a sequence of i.i.d.~random variables fulfilling
$\Prob(\varepsilon_1=1)=\Prob(\varepsilon_1=-1)=1/2$. Then
\begin{align}
    \label{eq:radnorm_definition}
    \radnorm{(x_1,\dots,x_n)}nE
    \equalDef  \E\norm[\big]{\isum \varepsilon_i x_i}_E \, , \qquad \qquad x_1,\dots,x_n\in E
\end{align}
defines another norm on $E^n$, which we sometimes call the \emph{Rademacher norm}.
Here we note that in the literature one also considers $p$-norms on the right hand side, but due to Kahane's inequality, see e.g.~ \cite[Thm.~11.1]{DiJaTo95},
the resulting norms are    equivalent   to \eqref{eq:radnorm_definition} with constants independent of $n$. 

Now, the basic idea of type and cotype is to compare sequence norms and Rademacher norms uniformly in $n$.
Namely, for an operator $A\in \calL(E,F)$ we define
	\begin{align}
		\label{eq:def_typenorm}
		\oldtypenorm A{\wasoncered{2}}n
		\equalDef
		\sup_{\substack{n\in\N\\{(x_1,\dots,x_n)\in E^n\setminus \set 0}}} \frac{\radnorm {(Ax_1,\dots,Ax_n)} nF}{ \ellpnorm {(x_1,\dots,x_n)} {\wasoncered{2}}nE}
		.
	\end{align}
and we say that $A$ is of \emph{type 2} if $\oldtypenorm A{\wasoncered{2}}n < \infty$. Analogously, we write
	\begin{align}
		\label{eq:def_cotypenorm}
		\oldcotypenorm A{\wasoncered{2}}n
		\equalDef
		\sup_{\substack{n\in\N\\{(x_1,\dots,x_n)\in E^n\setminus \set 0}}} \frac{ \ellpnorm {(Ax_1,\dots,Ax_n)} {\wasoncered{2}}nF}{\radnorm {(x_1,\dots,x_n)} nE}
	\end{align}
and we say that $A$ is of \emph{cotype 2}, if $\oldcotypenorm A{\wasoncered{2}}n <\infty$.
It is straightforward to verify that $\oldtypenorm \cdot {\wasoncered{2}}n$ defines a norm on the space of all type 2 operators $A\in \calL (E,F)$ and an analogous statement is true for $\oldcotypenorm \cdot{\wasoncered{2}}n$ on the space of
all cotype 2 operators $A\in \calL (E,F)$.
Moreover, if if we have Banach spaces $E_0$ and $F_0$, as well as bounded linear operator $S:E_0\to E$ and $T:F\to F_0$,
then some simple calculations show the following two inequalities, which describe an ideal property of type and cotype 2 operators:
\begin{align} \label{eq:type-2-ideal-norm}
\oldtypenorm {T\!AS}{\wasoncered{2}}n  &\leq \snorm T \oldtypenorm A{\wasoncered{2}}n  \snorm S\, , \\ \label{eq:cotype-2-ideal-norm}
\oldcotypenorm {T\!AS}{\wasoncered{2}}n&\leq \snorm T \oldcotypenorm A{\wasoncered{2}}n \snorm S\, .
\end{align}
Finally, we say that a Banach space $E$ is of type 2 or of cotype 2, if the identity map $\id_E:E\to E$ is of type 2, respectively cotype 2.
In this respect we also need to recall that Hilbert spaces $H$ are of both type 2 and cotype 2.
For more information on type and cotype we refer to \cite[Chapter 11]{DiJaTo95}.

\donee

Let $(X,\met)$ be a metric space. Then,
we denote the closed ball of radius $\d$ with center $x\in X$ by $\ball \d x$.
Moreover, given a $\d>0$,
a sequence $x_1,\dots,x_n\in X$ is called a $\d$-packing in $X$, if $\met(x_i,x_j)\ge \d$ holds  for any $i\ne j$. The \emph{packing numbers } of $X$ are defined as
\begin{align*}
	\paceuclid{X}{\d}\equalDef \pac{X}{\met}{\d} \equalDef \sup\{n\in\N\mid \text{ there exists a }\d\text{-packing } x_1,\dots, x_n\in X \}
\end{align*}
for all $\d>0$.
It is well known that for bounded $X\subset \Rd$ with non-empty   interior there exist
constants $0<c\le C$ such that
\begin{align}\label{eq:cov-num-Rd}
c \d^{-d} \le \pac\d {\met}{X} \le C\d^{-d}\, , \qquad \qquad \d\in (0,1].
\end{align}

\donee
\section{General Results}\label{sec:general}

\donee

In this section we present several results investigating the situation $E\subset H\subset F$,
where $E$ and $F$ are proper BSFs on   $X$ and $H$ is an RKHS on $X$. In particular, we derive both sufficient and necessary
conditions for the existence of such an RKHS  $H$.
We begin with a simple characterization in the case that $E$ is a \emph{closed subspace} of $F$.
In a nutshell it shows that in this case, no work for finding an RKHS $H$ is required.

\begin{proposition}\label{prop:closed-subspaces}
 Let $F$ be a proper BSF on $X$ and $E$ be a closed subspace of $F$. Then $E$ is a proper BSF and
 the following statements are equivalent:
 \begin{enumerate}
  \item There exists an RKHS $H$ with $E\subset H\subset F$.
  \item $E$ is isomorphic to a Hilbert space.
 \end{enumerate}
Moreover, in this case we have $E\hookrightarrow H \hookrightarrow F$.
\end{proposition}

Since neither $c_0$ nor $\ell_\infty$ are isomorphic to a Hilbert space we immediately
conclude by \eqref{eq:linfty-inclusion} and
Proposition \ref{prop:closed-subspaces} with $F\equalDef \ell_\infty$ that there exists no RKHS $H$ with bounded
kernel such that $c_0 \subset H$ or $\ell_\infty \subset H$.
Similarly, if $(X,d)$ is a compact metric space such that $E\equalDef C(X)$ is infinite dimensional, then this space
is not isomorphic to a Hilbert space, and by considering $F\equalDef\ell_\infty(X)$ in Proposition \ref{prop:closed-subspaces}, we conclude that there is no
RKHS $H$ with bounded kernel and $C(X) \subset H$.

\donee

Given a proper BSF $E$ on $X$ and a non-empty subset $Y\subset X$, we can consider the restricted functions $f_{|Y}:Y\to \R$
for all $f\in E$. The set $E_{|Y}$ of such restrictions is again a BSF, see Lemma \ref{lem:restricted-BFS} for a formal statement.
Moreover, 
we note that the restriction $H_{|Y}$ of an RKHS $H$ on $X$ with kernel $k$ is again an RKHS. Its  kernel
is given by restricting $k$ to $Y\times Y$, that is by $k_{|Y\times Y}$. This is a direct consequence of the
fundamental theorem of RKHS, see \cite[Thm.~4.21]{StCh08} and \cite[Ex.~4.4]{StCh08}.

The following proposition shows that inclusions $E\subset F$ are preserved when applying the same restriction operator
to both spaces $E$ and $F$.

\begin{proposition}\label{prop:restricted-inclusions}
 Let $E$ and $F$ be proper BSFs on $X$ with $E\subset F$. Then for all non-empty  $Y\subset X$ we have
 \begin{align*}
  E_{|Y} \hookrightarrow F_{|Y}
 \end{align*}
 with $\snorm{\id:  E_{|Y} \to F_{|Y}  } \leq \snorm{\id:E\to F}$.
\end{proposition}

To illustrate Proposition \ref{prop:restricted-inclusions}, we assume that we have two proper BSF $E$ and $F$ on  $X$ with
$E\subset F$ and we seek an RKHS $H$
with $E\subset H\subset F$. If there exists such an $H$, then Proposition \ref{prop:restricted-inclusions} ensures
\begin{align}\label{eq:restricted-inclusion}
E_{|Y} \hookrightarrow H_{|Y} \hookrightarrow F_{|Y}\, ,
\end{align}
and in addition, $H_{|Y}$ is an RKHS. In other words, if we find a non-empty $Y\subset X$, for which there is no
RKHS $\tilde H$ on $Y$ with $E_{|Y} \subset \tilde H \subset F_{|Y}$, then there is also no RKHS $H$ on $X$ with $E\subset H\subset F$.
As a consequence, we only need to consider our Question \eqref{eq:central_question} on subsets $Y$ that we can suitably control.

\donee
The following definition is crucial for deriving both positive and negative answers to our Question \eqref{eq:central_question}. 
\begin{definition}\label{definition:two-fac}
	Let $E,F$ be Banach spaces and $A:E\to F$ be a bounded linear operator.  We say that $A$ is \emph{2-factorable} if there exist a Hilbert space $H$ and bounded linear operators $U:E\to H$ and $V:H\to F$  such
	the following \emph{2-factorization} of $A$ holds:
	 \begin{align}
	 	\label{eq:twofactorisation}
	 	\begin{tikzcd}
	 		E \arrow[rd,"U"] \arrow[rr, "A"] \pgfmatrixnextcell \pgfmatrixnextcell F \\
	 		\pgfmatrixnextcell H \arrow[ru,"V"] \pgfmatrixnextcell
	 	\end{tikzcd}
 		.
	 \end{align}
	Moreover, we define $\twofacnorm A  \equalDef \inf \norm U\norm V$, where the infimum runs over all 2-factorizations of $A$.
\end{definition}

It can be shown that %
$\twofacnorm \mycdot$ defines an operator ideal norm, see \cite[Thm.~7.1 in combination with p.~155]{DiJaTo95}  for details.
Namely,  if we have Banach spaces $E_0$ and $F_0$ and bounded linear operators $S:E_0\to E$ and $T:F\to F_0$,
and a 2-factorable operator $A:E\to F$, then $T\!AS:E_0\to F_0$ is also 2-factorable and we have
\begin{align*}
\twofacnorm {T\!AS} \leq \snorm T \twofacnorm A  \snorm S\, .
\end{align*}
With this observation, we almost immediately obtain the following lemma, which will play a central role in our analysis.

\begin{lemma}\label{lem:2fact-implies-twotype}
Let $E,F$ be Banach spaces and $A:E\to F$ be a bounded linear operator.  If $A$ is 2-factorable, then it is both of type 2 and cotype 2.
\end{lemma}

\donee

Note that if we have BSFs $E$ and $F$ and an RKHS $H$ with $E\hookrightarrow H\hookrightarrow F$, then
the inclusion map $\id :E\to F$ is obviously 2-factorable.
The following theorem shows that the converse implication is also true.

\begin{theorem}\label{theorem:2-fac_in_bfs} 
	Let $E$ and $F$ be  BFSs on some set $X$ with
	$E \hookrightarrow F$. Then this inclusion map is 2-factorable if and only if there exists an RKHS $H$ over $X$ such that
	\begin{align*}
		E\subset  H\subset  F\, .
	\end{align*}
	In this case we also have $E\hookrightarrow H\hookrightarrow F$.
\end{theorem}

In the following,
we will derive both sufficient and necessary conditions
that are in many cases easier to check than an abstract 2-factorization. In particular,
our sufficient conditions make it possible to avoid an explicit construction of a 2-factorization, while our necessary
conditions can be used to show that such a construction is impossible.
\donee
We begin with  two sufficient conditions.

\begin{theorem}\label{thm:suff-2-fact}
Let $E$ and $F$ be  BFSs on some set $X$ with $E\subset F$. Then  there exists an RKHS $H$ over $X$ such that
	\begin{align*}
		E\hookrightarrow H\hookrightarrow F
	\end{align*}
if one of the following two conditions are satisfied:
\begin{enumerate}
\item The inclusion map $\id: E\to F$ is 2-summing.
\item $E$ is of type 2 and $F$ is of cotype 2.
\end{enumerate}
\end{theorem}

In view of \emph{i)} recall that 2-summing, or more generally, $p$-summing operators have been extensively investigated in the literature, 
see e.g.~\cite{DiJaTo95}. In particular, is is known that 
\donee
for certain pairs of Banach spaces $E$ and $F$, every bounded linear operator
$A:E\to F$ is 2-summing. Indeed, if, for example
$E$ is a $\ca L_1$-space and $F$ is
a $\ca L_2$-space, see e.g.~\cite[p.~60]{DiJaTo95} for a definition,
then one can show with the help of Grothendieck's
inequality that every bounded linear operator $A:E\to F$ is 1-summing,
see e.g.~\cite[Thm.~3.1]{DiJaTo95}, and therefore also 2-summing, see
e.g.~\cite[Thm.~2.8]{DiJaTo95}.
Moreover, every Hilbert space is an $\ca L_2$-space, and since
the 2-summing operators form an operator ideal, we   directly obtain
the following result, which shows that \emph{i)} of \cref{thm:suff-2-fact}
is sharp in some cases.

\begin{lemma}
Let $E$ and $F$ be  BFSs on some set $X$ with $E\subset F$. If $E$ is an $\ca L_1$-space
and
there exists an RKHS $H$ over $X$ with
$E\subset  H\subset  F$, then the inclusion map $\id:E\to F$ is 2-summing.
\end{lemma}

We note that a similar result holds, if $E$ is an $\ca L_\infty$-space
thanks to another application of Grothendieck's inequality, see e.g.~\cite[Thm.~3.7]{DiJaTo95}.
In addition, if $E$ is a subspace of an $\ca L_p$-space for some $p\in [1,2]$, then
every bounded linear $A:E\to F$ that is
$q$-summing for some $q>2$ is also  $2$-summing. For such $E$,
we can thus replace \emph{i)} of \cref{thm:suff-2-fact} by the $q$-summability of
$\id: E\to F$. Similarly, if $E$ is an $\ca L_p$-space for some $p\in [2,\infty]$
and $1/q \geq 1/2 - 1/p$, then every $q$-summing operator $A:E\to F$ is 2-factorizing, see
e.g.~\cite[p.~168]{DiJaTo95}.

The assumption in \emph{ii)} of \cref{thm:suff-2-fact} can be slightly relaxed. Indeed, if both $E'$ and $F$ 
have cotype 2 and one of these spaces has the so-called approximation property, then \cite{Pisier80a}, see also \cite[Thm.~4.1]{Pisier86}
has shown that every bounded linear operator $A:E\to F$ is 2-factorable. 
Here we also refer to \cite[Thm.~8.17]{Pisier86} 
for another relaxation.
Moreover, \cite[Thm.~3.4]{Pisier86} shows that every operator $A$ that has a factorization $A= VU$, where $U$ is of type 2 and $V$
is of cotype 2, is 2-factorable,  and since Hilbert spaces are of type 2 and cotype 2, the converse implication is obviously also true,   using the ideal property of type 2 and cotype 2.  
In this direction we also note that \cite{Leung90a} 
has found examples of Banach spaces $E$ and $F$ such that neither $E'$ nor $F$ have cotype 2, 
but every bounded linear operator $A:E\to F$ is 2-factorable. However,
these spaces are not 
BSFs. In addition, the embeddings $\ell_1\hookrightarrow \ell_2\hookrightarrow \ell_\infty$ show that these (co)type 2 assumptions on $E$ and $F$ 
are in general \emph{not} necessary for a positive answer to our Question \eqref{eq:central_question}. Moreover,  our  
results on H\"older spaces, see  
\cref{thm:abstract_holder_twofac_thm} and \cref{lemma:besov_c0_twofac}, show that in the absence of (co)type 2 assumptions on $E$ and $F$  
both positive and negative results are possible even for natural embeddings within the same scale of spaces. 
In contrast, the following straightforward corollary shows that the type 2 and cotype 2 of $\id:E\to F$
is a necessary condition for
$E\subset  H\subset  F$.
 
\donee

\begin{corollary}\label{cor:nec-2-fact}
Let $E$ and $F$ be  BFSs on some set $X$ with $E\subset F$. If there exists an RKHS $H$ over $X$ with
$E\subset  H\subset  F$, then the inclusion map $\id:E\to F$ is of both type 2 and cotype 2.
\end{corollary}

\done{}

\section{Applications to Various Concrete Spaces}\label{sec:examples}

We have seen in Theorem \ref{theorem:2-fac_in_bfs} that our Question \eqref{eq:central_question} is characterized by the 2-factorability of the embedding $E\hookrightarrow F$. 
In this section, we thus investigate the 2-factorability of embeddings $E_\theta\hookrightarrow F_{\theta'}$ in various parameterized families of spaces of ``functions'', 
establishing both \emph{positive} and \emph{negative results}. The positive results explicitly state 2-factorizations given sufficiently well-behaved parameters $\theta,\theta'$, which
we  straightforwardly derive from known embedding theorems. 
Negative results are derived from investigating type 2 and cotype 2, which are required for 2-factorability as stated in \cref{lem:2fact-implies-twotype}.

\donee

\subsection{Warm-up:  \texorpdfstring{$\ell_p$-spaces and $L_p$-spaces}{LG}}

In this subsection, we quickly investigate inclusions between $\ell_p$-spaces, respectively $L_p$-spaces. 
Here, our focus lies on explaining the common strategy for deriving negative results rather than on 
deriving particularly interesting new results. This common strategy will then applied to more interesting scenarios in the subsequent 
subsections.

We begin with the most simple example of a family of spaces for which we investigate 2-factorability, namely $\ell_p$-spaces.
Recall that these spaces are proper BSFs, and hence the following lemma also precisely answers our Question \eqref{eq:central_question} for this 
family of BSFs.

\begin{lemma}\label{lemma:analyzing_lp_spaces}
	Let $1\le p\le q\le \infty$. Then the embedding 
	\begin{align*}
		\ell_p\naturalsequence \hookrightarrow \ell_q\naturalsequence
	\end{align*}
	is 2-factorable if and only if  $p\le 2\le q$. In this case it 2-factorizes over the RKHS $\ell_2$ by
	$\ell_p\naturalsequence\hookrightarrow \ell_2\naturalsequence\hookrightarrow \ell_q\naturalsequence$.

\end{lemma}

\begin{proof}
    Clearly, we only 
 need to show that $p\le 2\le q $ is a necessary requirement for 2-factorability.
    To this end, let $e_1,e_2,\ldots \in \ell_p$ be 
	  the sequence of unit vectors.
	For $n\geq 1$  we then have
	\begin{align*}
		\radnorm{(e_1,\dots,e_n)}{n}{\ell_p\naturalsequence}
		&= \E_{\radsequence{\varepsilon}}
		\norm[\big] {
		\isum \varepsilon_i e_i }
_{\ell_p\naturalsequence} 
		= n^{1/p}
		\nonumber
\intertext{and}
		\elltwonorm{(e_1,\dots,e_n)}{n}{\ell_q\naturalsequence}&= \bra[\Big]{\isum \norm{e_i}_{\ell_q\naturalsequence}^2 }^{1/2} =n^{1/2}.
	\end{align*}
	Now, if $\ell_p\naturalsequence\hookrightarrow \ell_q\naturalsequence$ is 2-factorable, then it is of type 2 by \cref{lem:2fact-implies-twotype}, and hence we obtain 
	\begin{align*}
		\pretypenorm{\ell_p\naturalsequence\hookrightarrow \ell_q\naturalsequence}{2}{n} 
		\ge \sup_{(x_1,\dots,x_n)\in \ell_\pp\naturalsequence^n\setminus\set 0} \frac{\elltwonorm{(x_1,\dots,x_n)}{n}{\ell_q\naturalsequence} }{\radnorm{(x_1,\dots,x_n)}{n}{\ell_p\naturalsequence}}
		\ge \frac{\elltwonorm{(e_1,\dots,e_n)}{n}{\ell_q\naturalsequence} }{\radnorm{(e_1,\dots,e_n)}{n}{\ell_p\naturalsequence}}
		=  n^{1/2-1/p}
	\end{align*}
	for all $n\geq 1$. Since $\pretypenorm{\ell_p\naturalsequence\hookrightarrow \ell_q\naturalsequence}{2}{n}<\infty, $
	this implies $p\leq 2$.
Using $\precotypenorm{\ell_p\naturalsequence\hookrightarrow \ell_q\naturalsequence}{2}{n}  < \infty$
we analogously obtain   $q\ge 2$.
\end{proof}

The next lemma investigates $L_p$-spaces, where for the sake of simplicity we only consider the domain $[0,1]^d$ and the Lebesgue 
measure. Here we note that these spaces are, of course, not BSFs, but the common strategy becomes more visible 
than in the previous example. 

\begin{lemma}\label{lemma:analyzing_Lp_spaces}
	Let $\Om\equalDef[0,1]^d$ and $1\le \qq\le \pp\le\infty$. Then, the embedding $\embe{\lpspace \pp\Om}{\lpspace \qq\Om }$ is 2-factorable if and only if $\qq\leq 2\leq \pp$. In this case we have $\lpspace \pp\Om\hookrightarrow\lpspace2 \Om\hookrightarrow\lpspace \qq\Om$.
\end{lemma}

\begin{proof}
	Again, we only need to show that 2-factorability implies 
	$\pp\ge 2\ge \qq$. To this end, let $n\in\N$ and $m\equalDef n^d$. Moreover, let $A_1,\dots, A_{m}$ be the partition of $[0,1]^d$ into $m$ cubes of side length $1/n$.
	Then, we obtain
	\begin{align*}
		\radnorm{(\mathds 1_{A_1},\dots,\mathds 1_{A_{m}}) }{m}{\lpspace{\pp}{\Om}}
		&= \E_{\radsequence{\varepsilon}}\norm[\big] {\sum_{i=1}^{m}\varepsilon_i \mathds 1_{A_i} }_{\lpspace \pp\Om} 
		= 1
		\intertext{and}
		\nonumber
		\elltwonorm{(\mathds 1_{A_1},\dots,\mathds 1_{A_{m}}) }{m}{\lpspace{\qq}{\Om}}&= \bra[\bigg]{\sum_{i=1}^{m}\norm{\mathds 1_{A_i}}_{\lpspace{\qq}{\Om}}^2 }^{1/2} =n^{d(1/2-1/\qq)}.
	\end{align*}
	Now, if $\embe{\lpspace \pp\Om}{\lpspace \qq\Om }$ is 2-factorable, then it is of type 2 by \cref{lem:2fact-implies-twotype}, and hence we obtain 
	\begin{align*}
		\typenorm{\embe{\lpspace \pp\Om}{\lpspace \qq\Om }}
		\ge& \frac{
			\elltwonorm{(\mathds 1_{A_1},\dots,\mathds 1_{A_{m}}) }{m}{\lpspace{\qq}{\Om}}
			}{
			\radnorm{(\mathds 1_{A_1},\dots,\mathds 1_{A_{m}}) }{m}{\lpspace{\pp}{\Om}}
			}
		 = n^{d(1/2-1/\qq)}.
	\end{align*}
	for all $n\geq 1$. 
	Since $\typenorm{\embe{\lpspace \pp\Om}{\lpspace \qq\Om }}<\infty$, this implies  $\qq\leq 2$. 
	Analogously, we obtain   $\pp\ge 2$ using $\cotypenorm{\embe{\lpspace \pp\Om}{\lpspace \qq\Om }}<\infty$.
\end{proof} 

It is easy to generalize \cref{lemma:analyzing_Lp_spaces} to atom-free finite measure spaces $(\Om,\mu)$. Since \cref{lemma:analyzing_Lp_spaces}
is mostly stated to illustrate our common proof strategy, we omit the details.

\donee

\subsection{Spaces of H\"older Continuous Functions}\label{subsection:abstact_holder}

Let $(X,d)$ be a bounded metric space. 
In this section, we establish necessary requirements for the existence of an RKHS $H$ on $X$ that can be squeezed in between two fixed
H\"older spaces over $(X,\met)$, that is
\begin{align*}
	\mholspace{X}{\al}{\met}\hookrightarrow H \hookrightarrow \mholspace{X}{\be}{\met} \, ,
\end{align*}
where $0<\be\le \al \le 1$.
Here, our   result  will show that the existence of such an RKHS $H$
implies an upper asymptotic bound on the growth of the  packing numbers $\pac \d \met X$ of $(X,\met)$.
Moreover, we obtain a similar result if we replace $\mholspace{X}{\be}{\met}$ by $\ell_\infty(X)$.

Let us begin by recalling the definition of these spaces. To this end, 
	we denote the $\al$-H\"older norm of a function $f:X\to \R$ by
	\begin{align}
		\label{definition:holder_abstract_metric}
		\mholnorm{f}{\al}{\met}\equalDef 
		\max
		\left\{  
		\sup_{x\ne y\in X} \frac{\abs{f(x)-f(y)}}{\met^\al(x,y)}  , \sup_{x\in X} \abs{f(x)}
		\right\} \, ,
	\end{align}
	where   $\al\in (0,1]$.
	Clearly, we have $\mholnorm{f}{\al}{\met}<\infty$ if and only if $f$ is both $\a$-H\"older continuous and bounded.
	Moreover, we define the H\"older spaces by
	\begin{align*}
		\mholspace{X}{\al}{\met}\equalDef \{ f:X\to \R \mid \mholnorm f \al \met <\infty \} \, .
	\end{align*}
	It is well-known and easy to verify that $(\mholspace{X}{\al}{\met},\mholnorm{\mycdot}{\al}{\met})$ is a Banach space.
	In addition, since H\"older spaces consist of bounded functions,
	we can quickly check that
	$\mholspace{X}{\al}{\met}\hookrightarrow  \mholspace{X}{\be}{\met}$
	holds for all $0<\be\le \al \le 1$.
	Moreover, another routine check shows that it is also a proper BSF. By
	\cref{theorem:2-fac_in_bfs}, the existence of an intermediate RKHS $H$ is therefore
	equivalent to the 2-factorability of an embedding $\mholspace{X}{\al}{\met}\hookrightarrow \mholspace{X}{\be}{\met}$.

Finally, let us recall that the diameter of a metric space $(X,\met)$ is $\diam X \equalDef \sup_{s,t\in X}\met(s,t)>0$.
With these preparations our result reads as follows.

\begin{restatable}{theorem}{coveringnumberrestriction}\label{thm:abstract_holder_twofac_thm}
	Let $(X,\met)$ be a metric space and $0<\be\le \al \le 1$. If $X$ is connected and there exists an RKHS $H$ over $X$ with
	\begin{align}\label{eq:double-hoelder}
		\mholspace X \al \met \hookrightarrow H \hookrightarrow \mholspace X \be \met \, ,
	\end{align}
	then there exists a constant $C>0$ such that for all $0<\d< \min\{1, \diam X\}$ we have
	\begin{align}
		\label{eq:holder_packing_asymptotics}\pac\d {\met}{X}  \le  C\d^{-2(\al-\be)} \, .
	\end{align}
	Moreover, if there exists an RKHS $H$ over $X$ such that 
	\begin{align}\label{eq:hoelder-infty}
		\mholspace X \al \met \hookrightarrow H \hookrightarrow \ell_\infty(X) \, ,
	\end{align}
	then there exists a constant $C>0$, such that \eqref{eq:holder_packing_asymptotics} holds for $\b\equalDef 0$ and all $0<\d< \min\{1, \diam X\}$.
\end{restatable}

Let us consider a  bounded $X\subset \Rd$ with non-empty   interior.
By \cref{thm:abstract_holder_twofac_thm} and the packing number bound
\eqref{eq:cov-num-Rd}
we conclude that there is no RKHS $H$ satisfying \eqref{eq:double-hoelder}
if $\al-\be < d/2$.
Moreover, there is no RKHS $H$ satisfying \eqref{eq:hoelder-infty} if $\al< d/2$. 
Hence, \eqref{eq:central_question} is infeasible for $d> 2$ within the family of H\"older spaces.
We observe that the difference of the smoothness between the H\"older spaces in \eqref{eq:double-hoelder} 
needs to grow proportional to the underlying dimension in order to enable \eqref{eq:central_question}.

In \cref{subsubsection:isotropic_spaces}, \eqref{eq:triebelid-holder} will generalize H\"older spaces over $X$, allowing arbitrarily large non-integer smoothness. In \cref{thm:two_factorisability_of_general_triebel_besov_embedds} we will show that if $\al-\be >d/2$, then there exists an intermediate RKHS $H$ such that we have
\begin{align*}
	\holspace{\al}{X}\hookrightarrow H \hookrightarrow \holspace{\be}{X} \, 
\end{align*}
and in \cref{lemma:besov_c0_twofac} we will show that for $\al>d/2$ there exists an  intermediate RKHS $H$ such that
\begin{align*}
	\holspace{\al}{X}\hookrightarrow H \hookrightarrow C^0(X)
\end{align*}
holds.
Hence, the bounds derived in \cref{thm:abstract_holder_twofac_thm} are tight up to the border cases $\al-\be =d/2$, respectively $\al=d/2$.
For general metric spaces $(X,d)$ however, we do not know any condition that implies the existence of an RKHS $H$ satisfying
\eqref{eq:double-hoelder} or \eqref{eq:hoelder-infty}.

The requirement that $X$ is a \emph{bounded} subset of $\Rd$ is crucial to obtain positive answers to \eqref{eq:central_question}, as the following theorem shows.

\begin{theorem}\label{lem:unbounded_domain}
	Let $\Om\subset\Rd$ be open  and unbounded.
	Then there exists no RKHS $H$ on $\Om$ with bounded kernel such that $C^\infty(\Om) \subset H$.
\end{theorem}
By $C^\infty(\Om)$ we denote the space of smooth functions, given by the norm $\norm f_{C^\infty(\Om)}=\sup_{\al\in\N_0^d} \supnorm{\partial_\al f}$.

\subsection{Spaces of (Generalized) Mixed Smoothness}\label{subsubsection:sobolev_spaces_of_mixed_smoothness}

In this section we investigate our Question \ref{eq:central_question} for a scale of spaces that 
generalize both \emph{classical Sobolev spaces}
and \emph{Sobolev spaces of mixed smoothness}, see e.g.~\cite{AdFo03}, respectively  \cite{Sloan98}.
Since some of these spaces are not BSFs, we will focus on the 2-factorability of embeddings between such spaces, where
we recall that by \cref{theorem:2-fac_in_bfs} 2-factorability is equivalent to our initial Question \ref{eq:central_question} if the involved spaces are BSFs.

To introduce these spaces, we fix a finite, non-empty  set $A$ of multi-indices, i.e.~$\A\subset \N_0^d$. Then we say that $A$ is 
{coherent}, if for all $\a\in A$ and $\b\in \N_0^d$ with $\b\leq \a$ we have $\b\in A$, where $\be\le\al$ denotes the usual element-wise partial order.
Moreover,  we write $\munorm{\A}\equalDef \max\{\munorm\al:\al\in\A\}$, where 
 $|\a|_1 \equalDef \a_1+\dots+\a_d$ denotes  the sum of the entries of $\a\in \N_0^d$.

Given a bounded domain 
 $\Om\subset\Rd$  and a coherent $\A\subset \N_0^d$, we now define the set of $A$-times weakly differentiable functions by 
\begin{align}
	\label{eq:sobolev_auxspace}
	\preslospace{\A}{\Om}
	\equalDef \{ \fun\in \calL_0(\Om) \mid \text{the weak derivative }\partial_\al \fun \text{ exists for all }\al\in\A \}\, ,
\end{align}
where $\calL_0(\Om)$ denotes the space of all measurable $f:\Om\to \R$.
Moreover, we define the \emph{Sobolev space of generalized mixed smoothness $\A$}  as
$$
\slospace{\A}{\pp}{\Om}\equalDef
\{\fun\in\preslospace{\A}{\Om} : \slonorm \fun \A\pp\Om <\infty  \},
$$ 
$\text{ where the \emph{mixed Sobolev norm }of $\fun$ is given by }$
\begin{align}
\label{eq:mixed_sob_norm}
\slonorm\fun{\A}{\pp}{\Om}\equalDef \sup_{\al\in\A} \lpnorm{\partial_\al \fun}{\pp}{\Om}.
\end{align}
Note that for $s\in \N_0$ and $\A\equalDef\{\al\in\N_0^d: \munorm{\al}\le s\}$ the spaces $\slospace{\A}{\pp}{\Om}$ equal the classical Sobolev spaces $H^s_p(\Om)$
of integer smoothness $s$
modulo an equivalent norm. Here we note that  both our initial Question \ref{eq:central_question} and 2-factorability are invariant with respect to equivalent norms, and therefore 
these small differences between $\slospace{\A}{\pp}{\Om}$ and $H^s_p(\Om)$ do not affect our results.
Finally note that we obviously have $\munorm\A=s$.

If, for a fixed $s\in \N_0$, we define $\A\equalDef\{ \al\in\N_0^d: \al_1,\dots,\al_d \le s \}$, then $W^{\ss,\text{mix}}_\pp (\Om)\equalDef \slospace{\A}{\pp}{\Om}$ equals a
Sobolev space of mixed smoothness $s$, where again the latter spaces may have an equivalent norm.
Note that for such $A$ we have $\munorm\A=sd$.

In the following we  provide necessary parameter requirements for 2-factorability of embeddings between Sobolev spaces of generalized  mixed smoothness. 

\donee

\begin{theorem}\label{thm:mixed_twofac}
	Let $\Om\subset\R^d$ 
	be a bounded domain,
	 $\A,\B \subset \N_0^d$ be coherent sets, and $1\le \pone,\ptwo <\infty$ such that
	 the embedding $\embe{\slospace \A\pone\Om}{\slospace\B\ptwo\Om}$ exists.
	 We write  $\ss\equalDef \munorm\A$ and $\tt\equalDef\munorm \B$, and further assume that $\ss-\tt\ge d(1/\pone-1/\ptwo)$ holds.
	If the embedding $\embe{\slospace \A\pone\Om}{\slospace\B\ptwo\Om}$ is 2-factorable, then we have
	\begin{align}\label{thm:mixed_twofac-par-cond}
		\ss-\tt \ge \cone+\ctwo.
	\end{align}
\end{theorem}

Recall that 
for classical Sobolev spaces $\sobspace \ss \pone \Om$ and $\sobspace \tt\ptwo\Om$ of integer smoothness with  $\ss\ge \tt$,
it is asserted in \cite[p.~82, Remark]{RuSi96}, see also the discussion in \cite[p.~108ff]{AdFo03},
that the embedding $\embe{\sobspace \ss \pone \Om}{\sobspace \tt\ptwo\Om}$ 
exists if and only if $\ss-\tt\ge d/\pone-d/\ptwo$. 
Moreover, recall that for these spaces we have $\sobspace \ss p \Om = \slospace \A p\Om$, where $\A\equalDef\{\al\in\N_0^d: \munorm{\al}\le s\}$.
Since this gives $|A|_1=s$, \cref{thm:mixed_twofac} shows that if $\embe{\sobspace \ss \pone \Om}{\sobspace \tt\ptwo\Om}$ is 2-factorable
then \eqref{thm:mixed_twofac-par-cond} holds. 

In contrast, if we  consider Sobolev spaces of mixed smoothness $s$ respectively $t$ as introduced above, then the 2-factorability of
the embedding
$W^{\ss,\text{mix}}_{\pone} (\Om)\hookrightarrow W^{t,\text{mix}}_{\ptwo} (\Om)$
implies 
\begin{align*}
s-t \geq \pospart{1/\pone - 1/2} + \pospart{1/2 - 1/\ptwo}
\end{align*}
since the considered coherent sets $A$ and $B$ are of size  $\munorm\A=sd$ and  $\munorm\B=td$. 
Note that compared to classical Sobolev spaces, the dimension no longer appears in the necessary condition.

The next theorem investigates the case in which $\slospace\B\ptwo\Om$ is replaced by 
$\ell_\infty(\Om)$, that is we consider  $\embe{\slospace\A\pp\Om}  {\ell_\infty(\Om)}$, where the embedding 
is understood in the usual sense of Sobolev's embedding theorem.
Informally speaking this complements with the case
\cref{thm:mixed_twofac} with $B=\{0\}$ and $\ptwo = \infty$.

\begin{theorem}\label{thm:mixed_twofac_C-0}
	Let $\Om\subset\Rd$ be a bounded domain,  $\A\subset \N_0^d$ be a coherent set, and  $1\le \pp< \infty$ such that the
	embedding  $\embe{\slospace\A\pp\Om}  {\ell_\infty(\Om)}$ exists. We define $\ss\equalDef \munorm\A$ and assume $\ss\ge d/\pp$.
	If there exists an RKHS $H$ with 
	 ${\slospace\A\pp\Om} \subset H\subset {\ell_\infty(\Om)}$, then we have
	\begin{align*}
		\ss \ge \czero +d/2 \, .
	\end{align*}
\end{theorem}

\donee

	We can generalize \cref{thm:mixed_twofac} to \emph{mixed fractional smoothness} in the style of \emph{Slobodeckij spaces}, which we will use in \cref{def:slob_spaces}. A space $\slospace Ap\Om$ of mixed fractional smoothness is given by an integration index $p\ge 1$ and a \emph{coherent set of mixed smoothness} $A$. The elements $(\al,\th)$ of $\A$ consist of a multi-index $\al\in \N_0^d$ which ensures the coherent structure of $A$, and a parameter  $\th\in [0,1)$ that refines the smoothness of the derivative $\partial_\al$ to a fractional smoothness as in \eqref{eq:slobodeckij_core}.
	Defining $\munorm A\equalDef \max\{ \munorm\al +\th \mid (\al,\th) \in A\}$ we obtain analogously to \cref{thm:mixed_twofac} that for an embedding $\slospace \A \pone \Om \hookrightarrow \slospace \B \ptwo\Om$ the inequality 
	\begin{align*}
		\munorm\A-\munorm\B \ge \pospart{d/\pone -d/2} +\pospart{d/2 -d/\ptwo}
	\end{align*}
	is a necessary requirement for 2-factorability. For  Slobodeckij spaces  this result is sharp, as we will show in \cref{lemma:slobotwofac}.

\subsection{Embedding Integral Reproducing Kernel Banach Spaces into RKHSs}\label{subsubsection:RKBS}
Consider a single hidden layer neural network $f:\Rd\to\R$ with $n$ hidden neurons, that is, $f$ is of the form 
\begin{align*}
	f(x) = \sum_{i=1}^n \alpha_i  \sigma(\scal{x}{w_i} + b_i),
\end{align*}
where $\sigma:\R\to\R$ is an activation function, $w_i\in\R^d$ are weight vectors and $b_i\in\R$ are biases. This network's outputs can be interpreted as integrals by observing that for the  discrete 
signed measure 
\begin{align*}
	\mu = \sum_{i=1}^n \alpha_i \delta_{(w_i,b_i)}
\end{align*}
 on $\R^{d+1}$
the identity
\begin{align*}
	f(x) = \int_{\R^{d+1}} \sigma\bra{ \scal{(x,1)}{(w,b)}_{\R^{d+1}} } \dx \mu{(w,b)}
\end{align*}
holds for all $x\in \Rd$, where we write $(w,b)\in \R^{d+1}$ for the vector obtained by appending $b\in\R$ to $w\in\Rd$.
\donee
This perspective motivates the theory of \emph{integral reproducing kernel Banach spaces} (IRKBS), which can be used to describe such networks and their limits, see e.g.~\cite[Section 3.5]{BaDVRoVi23a}.

In this section we investigate Question \eqref{eq:central_question} for IRKBS, which are a  certain class of proper BSF.
Let us begin by recalling the definition of IRKBS.

\begin{definition}\label{def:irkbs}
Let $(X,\calA)$ be a measurable space,  $\Psi:X\times X \to\R$ be a measurable function, and $\calM$ be a Banach space of finite  signed  measures on $X$.
We call a  bounded and measurable function $\beta:X\to \R$ a  normalizing function of $\Psi$, if
\begin{align}
	\label{eq:smoothed_activation_bounded}
	\sup_{y\in X } \abs{  \Psi(x,y)\beta (y)} <\infty
\end{align}
holds for all $x\in X$. In this case, we define $f_{\mu,\Psi,\beta}:X\to \R$ by \donee
\begin{align}
	\label{eq:generalized_kernel_mean_embedding}
	f_{\mu,\Psi,\beta}(x) \equalDef \int_X  \Psi(x,y) \beta(y)\dx \mu(y)\, , \qquad\qquad x\in X
\end{align}
for each $\mu \in \calM$. Then the set $E_{\calM,\Psi,\beta}\equalDef \{f_{\mu,\Psi,\beta}\mid \mu\in\calM\}$ equipped with the norm 
\begin{align*}
\norm{f_\mu}_{E_{\calM,\Psi,\beta}}\equalDef \inf\{\norm\nu_{\calM}\mid \nu\in\calM, f_\nu=f_\mu\}
\end{align*}
is called an integral reproducing kernel Banach space.
\end{definition}

In the following, 
$\calM(X)$  denotes the space of finite signed measures on $X$ equipped with the total variation (TV) norm.
The next result, which is variant of \cite[Prop.~3.3]{BaDVRoVi23a}, shows that under natural conditions
$E_{\calM,\Psi,\beta}$ becomes a proper BSF.

\begin{proposition}\label{prop:irkbs}
	Let $X$ be a measurable space, $\Psi:X\times X\to\R$ be a measurable function and $\beta:X\to\R$ be a normalizing function of $\Psi$.
	Let $\calM$ be a Banach space of finite  signed  measures on $X$ with  $\calM \hookrightarrow \calM(X)$.
	Then, $( E_{\calM,\Psi,\beta} ,\norm\mycdot_{E_{\calM,\Psi,\beta}})$ is a proper BSF.
\end{proposition}

In the following, we would like to reduce our considerations to the case $\b=1$ without loosing generality. To explain this, 
let $(X,\calA)$ be a measurable space, $\b:X\to \R$ be bounded and measurable, and  $\calM$ be a Banach space of  signed  measures on $X$.
For $\mu \in \calM$, we 
define a new finite, signed measure $\beta\mu$ on $X$ 
by 
\begin{align*}
(\beta\mu)(A)\equalDef\int_A \beta(y) \dx \mu(y)\, , \qquad\qquad A\in \calA.
\end{align*}
Moreover, we write $\beta\calM \equalDef\{ \beta\mu \mid \mu\in\calM\}$ and 
\begin{align*}
\norm{\mu}_{\beta\calM} \equalDef \inf\{ \norm \nu_{\calM} \mid \nu\in\calM, \mu =\beta\nu\} \, .
\end{align*}
It can be quickly verified, that $(\beta\calM, \norm{\cdot}_{\beta\calM})$ is a Banach space of finite signed measures. 
Moreover,  the embedding $\calM\hookrightarrow\calM(X)$ and the boundedness of $\beta:X\to\R$ imply $\beta\calM\hookrightarrow \calM(X)$.
Under the assumptions of \cref{prop:irkbs} we further have 
\begin{align}
	\label{eq:moving_smoothing_function_fun}f_{\mu,\Psi,\beta} &= f_{\beta\mu,\Psi,1}
\end{align}
and also %
\begin{align} \nonumber
	\norm{f_{\mu,\Psi,\beta}}_{E_{\calM,\Psi,\beta}}
	= \inf\{ \norm\nu_{\calM} \mid \nu\in\calM, f_{\nu,\Psi,\beta}=f_{\mu,\Psi,\beta}\}
	&= \inf\{ \norm\nu_{\calM} \mid \nu\in\calM, f_{\beta\nu,\Psi,1} =f_{\beta\mu,\Psi,1}\} \\ 	\nonumber
	&=\inf\{ \norm{\beta\nu}_{\beta\calM} \mid f_{\beta\nu,\Psi,1} =f_{\beta\mu,\Psi,1}\}\\ \label{eq:moving_smoothing_function_norm}
	&=
	\norm{f_{\beta\mu,\Psi,1}}_{E_{\beta\calM,\Psi,1}}.
\end{align} 
In other words, 
\eqref{eq:moving_smoothing_function_fun} and \eqref{eq:moving_smoothing_function_norm} make it possible to work with the constant normalizing function $\beta=1$ by adapting the considered Banach space of signed  measures. 

Having introduced IRKBSs, our next goal is to investigate under which assumptions these spaces can be embedded into an RKHS. 
To this end, we assume that $\Psi:X\times X\to\R$  is a symmetric function that can be written as a 
difference of   two kernels  $k_1$ and $k_2$ on $X$, that is
\begin{align*}
\Psi=k_1-k_2 \, .
\end{align*}
In the literature, this is 
 known as a \emph{positive decomposition} of $\Psi$ and it is discussed in e.g.~\cite{DaMoOr23} and the references mentioned therein. 
 Moreover, positively decomposable functions are the foundation of reproducing kernel Krein spaces (RKKSs), a generalization of RKHSs, see e.g.~\cite{OnMa04}. 

The following lemma investigates Question \eqref{eq:central_question} for IRKBS $E_{\calM,\Psi,1}$ in cases where $\Psi$ is positively decomposable.

\begin{lemma}\label{lem:difference_of_kernels}
	Let $X$ be a measurable space and $\calM$ be a Banach space of finite  signed  measures on $X$ with $\calM\hookrightarrow \calM(X)$.
	Moreover, let $k_1,k_2$ be measurable kernels on $X$ with RKHSs $H_1,H_2$, such that 
	 $\Psi\equalDef k_1-k_2$  is bounded. %
	Then, the IRKBS $E_{\calM,\Psi,1}$ is a proper BSF. 
	If, in addition, we have 
	\begin{align}\label{eq:requirement_rkhs_decomp}
		H_1,H_2\subset \calL_1(\mu) 
	\end{align}
	 for all $\mu\in\calM$,
	then the RKHS $H=H_1+H_2$ of the kernel $k\equalDef k_1+k_2$ fulfils 
	\begin{align*}
		E_{\calM,\Psi,1} \hookrightarrow H.
	\end{align*}
\end{lemma}

An  easy way to check \eqref{eq:requirement_rkhs_decomp} is provided by 
\cite[Thm.~4.26]{StCh08}, which  shows that if $k$ is a measurable kernel on $X$ with RKHS $H$ and $\mu$  is a finite signed measure on $X$ with  
	\begin{align*}
			\norm k_{L_1(\mu)} \equalDef \int_X \sqrt{k(x,x)} \, d \abs \mu (x) <\infty \, ,
	\end{align*}
	where $\abs \mu \equalDef \mu_+-\mu_-$ denotes the total variation of $\mu$, then  $H\subset \calL_1(\mu)$ follows.  %
	
To adapt this to our situation, we first note that \eqref{eq:requirement_rkhs_decomp} is obviously equivalent to 
$H_1+H_2 \subset\calL_1(\mu)$ for all  $\mu\in\calM$. Since $k_1+k_2$ is the kernel of $H_1+H_2$, we thus see that 
\eqref{eq:requirement_rkhs_decomp} is satisfied if 
	\begin{align}
			\label{lem:rkhs_integrability-new}
			\int_X \sqrt{k_1(x,x) + k_2(x,x)} \dx \mu (x) <\infty  \, , \qquad \qquad \mu\in\calM.
		\end{align}
For example, if $k_1$ and $k_2$ are bounded, \eqref{lem:rkhs_integrability-new} holds, and thus also \eqref{eq:requirement_rkhs_decomp}. As a consequence, 
we then have the embeddings  $E_{\calM,\Psi,1} \hookrightarrow H_1+H_2 \hookrightarrow \ell_\infty(X)$.

To illustrate \cref{lem:difference_of_kernels} and its crucial assumption  \eqref{eq:requirement_rkhs_decomp} in a concrete setting, 
we fix a real sequence
 $(\lambda_i)_{i\ge 0}$ and consider the power series 
 \begin{align*}
 \sigma_+(t) & \equalDef \sum_{i=0}^\infty \pospart{\lambda_i} t^i  \, ,\\
 \sigma_-(t) &\equalDef \sum_{i=0}^\infty \negpart{\lambda_i} t^i
 \end{align*}
as well as their convergence radii $R_+$ and $R_-$. Moreover, we write 
 $R\equalDef\min\{R_+,R_-\}$ and assume $R>0$.

 For $X\subset \mathring{B}(0,\sqrt R)$ we define  $k_1, k_2: X\times X\to\R$ by $k_1(x,y)\equalDef \sigma_+(\scal xy)$ and $k_2(x,y)\equalDef \sigma_-(\scal xy)$. Then, $k_1$ and $k_2$ are  continuous and \cite[Lemma 4.8.]{StCh08} 
 shows that $k_1$ and $k_2$ are kernels on $X$ since all coefficients are non-negative.
We further define $\Psi:X\times X\to \R$ by $\Psi\equalDef k_1-k_2$, that is, 
\begin{align*}
\Psi(x,y)=\sum_{i=0}^\infty \lambda_i \scal xy^i \, , \qquad\qquad x,y\in X.
\end{align*}Let us assume for simplicity that $\Psi$ is bounded, so that we may
 choose $\beta \equiv 1$ as a normalizing function of $\Psi$. 

Then, 
\cref{lem:difference_of_kernels} shows that the IRKBS $E_{\calM(X),\Psi,1}$ is a proper BSF. 
To obtain a positive answer to Question \eqref{eq:central_question} from \cref{lem:difference_of_kernels}
it remains to check \eqref{eq:requirement_rkhs_decomp}.
Here we first note that the sufficient condition \eqref{lem:rkhs_integrability-new} for \eqref{eq:requirement_rkhs_decomp} reads as 
\begin{align*}
	\sqrt{\sum_{i=0}^\infty \abs{\lambda_i} \norm \mycdot^{2i}} \in \calL_1(\mu) \, , \qquad \qquad \mu\in\calM.
\end{align*}
Unfortunately, this condition is not necessary. To derive a necessary condition, we first note that \eqref{eq:requirement_rkhs_decomp}
implies $k_i(\mycdot,x)  \in \calL_1(\mu)$ for all $x\in X$, $\mu\in\calM$, and $i=1,2$. In our situation, the latter reads as 
\begin{align}
	\label{eq:integrability_kernels}
	\sum_{i=0}^\infty \abs{\lambda_i} \scal \mycdot x^i \in\calL_1(\mu)
\end{align}
for all $x\in X$ and  $\mu\in\calM$.

\secondd

\donee

Finally, the following example shows that the condition \eqref{eq:smoothed_activation_bounded} ensuring  that $E_{\calM,\Psi,1}$   is a proper BSF
is in general substantially weaker than the condition \eqref{eq:requirement_rkhs_decomp} ensuring a surrounding RKHS.

\begin{example}
	Let $X=\Rd$, $\sigma:\R\to \R$ be given by
	\begin{align*}
		\sigma(t) &= \cos(t) = \sum_{i=0}^\infty (-1)^i \frac{t^{2i}}{(2i)!} \, ,
	\end{align*}
	and $\Psi:X\times X\to\R$ be given by
	$		\Psi(x,y) =\sigma(\scal xy).
	$
	Then, $\Psi$ is bounded and for $\calM := \calM(X)$ we know by \cref{lem:difference_of_kernels} that 
	the IRKBS $E_{\calM,\Psi,1}$ is a proper BSF. Moreover, our discussion above shows that 
	a positive decomposition of $\Psi$ is given by the kernels
	\begin{align*}
	k_1(x,y)= \sum_{i=0}^\infty \frac {t^{4i}}{(4i)!} \qquad  \mbox{ and } \qquad  k_2(x,y)= \sum_{i=0}^\infty \frac{t^{4i+2}}{(4i+2)!}. 
	\end{align*}
	By \eqref{eq:integrability_kernels} we thus see that we need at least 
	\begin{align*}
		\cosh(\scal \mycdot x) 
		\in \calL_1(\mu)
	\end{align*}
	 for all $\mu\in \calM, x\in X$ in order to apply \cref{lem:difference_of_kernels}. %
	 Consequently, we can only find an RKHS $H$ with  $E_{\calM,\Psi,1}\hookrightarrow H$ 
	 with the help of \cref{lem:difference_of_kernels}, if we substantially reduce the space $\calM$ of 
	 considered measures.
\end{example}

\donee

\subsection{Triebel-Besov-Lizorkin spaces}\label{subsubsection:isotropic_spaces}

In this section we consider Triebel-Lizorkin spaces and Besov spaces. To quickly introduce those spaces, let
us fix a  bounded domain $\Om\subset\Rd$  with smooth boundary \cite[Sect.~2.4.1, Def.~1]{RuSi96}.
We denote
\emph{Triebel-Lizorkin spaces} by $\triebelspace{\ss}{\pp}{\qq}{\Om}$,
where $s\geq 0$ is the \emph{smoothness},
$p\in [1,\infty)$ is the \emph{integration index},
and $q\in [1,\infty]$  is the \emph{fein index}.
Moreover,
\emph{Besov spaces} are denoted as $\besovspace{\ss}{\pp}{\qq}{\Om}$ where $s\geq 0$ and $p,q\in [1,\infty]$.

In the following, we say that $\spaceone{\ss}{\pp}{\qq}{\Om}$ is a
Besov-Triebel-Lizorkin space, if either
$\spaceone{\ss}{\pp}{\qq}{\Om} = \triebelspace{\ss}{\pp}{\qq}{\Om}$
with $p<\infty$ or
$\spaceone{\ss}{\pp}{\qq}{\Om} = \besovspace{\ss}{\pp}{\qq}{\Om}$
with $p\in [1,\infty]$.

\donee

Recall from e.g.~\cite[Sect.~2.2.4]{RuSi96} that
for $\ss>d/\pp$, the spaces $\triebelspace\ss\pp\qq\Om$ and $\besovspace\ss\pp\qq\Om$ can be continuously embedded into $C^0(\Om)$ and
consequently, they are proper BSFs.  
Nonetheless,
we will also investigate the case $\ss\le d/\pp$, where we can still sharply answer the question of 2-factorability.
Finally, some additional  embedding theorems for Besov-Triebel-Lizorkin spaces are recalled in
\cref{theorem:triebels_embeddings}.

Many well known function spaces are special cases of Besov-Triebel-Lizorkin spaces
\cite[Prop.~2.3.5]{Triebel83}, especially the following identities hold up to norm-equivalence:

The fractional Sobolev spaces $\sobspace{\ss}{\pp}\Om$, which are sometimes also referred to as Bessel potential spaces, can be identified via
		\begin{align}
			\label{eq:triebelid-sob}\sobspace{\ss}{\pp}\Om= \triebelspace{\ss}{p}{2}\Om\, , \qquad\qquad s\geq 0, p\in (1,\infty).
		\end{align}
Here we recall that $\sobspace{\ss}{2}\Om$ is a \emph{Hilbert space} which can be quickly inferred from their definition with the help
of the Fourier transformation, see e.g.~\cite[p.~13]{RuSi96}. Combining this with the above mentioned embedding into $C^0(\Om)$,
the spaces $\sobspace{\ss}{2}\Om$ are RKHSs whenever
$\ss>d/2$.

For $s>0$ with $s\not \in\N$, the H\"older spaces $\holspace \ss\Om$ of
$\floor{\ss}$-times differentiable, bounded functions, for which all these derivatives are bounded and the derivatives of order $\floor{\ss}$ are  $\ss-\floor\ss$-H\"older continuous,
can be identified via
\begin{align}
	\label{eq:triebelid-holder}
\holspace{\ss}{\Om}= \besovspace \ss \infty \infty \Om.
\end{align}
Note that for $\ss\in (0,1)$, these spaces coincide with the H\"older spaces discussed in \cref{subsection:abstact_holder}.
In contrast to \cref{subsection:abstact_holder}, however, we can now also deal with  smoothness $s>1$.

For $\ss>0$ and $p\in [1,\infty)$,
 the Slobodeckij spaces $\slospace{\ss}{\pp}{\Om}$, whose definition is recalled in \cref{def:slob_spaces} can be identified via
		\begin{align}
			\label{eq:triebelid-slobo}
			\slospace{\ss}{\pp}{\Om} = \triebelspace{\ss}{\pp}{\qq} \Om,
			\qquad\qquad
			\text{ where }\qq=\begin{cases}
				2, &\text{ if $\ss\in\N$ and $p>1$,}\\
				\pp, & \text{ if $\ss\not\in \N$.}
			\end{cases}
		\end{align}

Besov spaces and Triebel-Lizorkin spaces coincide when their integration and fein index are equal, that is
\begin{align}
	\label{eq:triebel-besov-identity}
	\besovspace \ss \pp\pp\Om =\triebelspace \ss\pp\pp\Om, \qqqquad \pp<\infty.
\end{align}
This is the only scenario in which Besov spaces and Triebel-Lizorkin spaces are equal, see \cite[2.3.9]{Triebel83}.

\donee

\donee 

Our first result in this section investigates  2-factorability of embeddings between  Slobodeckij spaces. Modulo limiting cases, it
provides a characterization.

		\begin{restatable}{theorem}{slobotwofac}\label{lemma:slobotwofac}\label{corollary:parameter_reqs_in_slob_spaces}
			Let $\Omega\subset\R^d$ be a  bounded  domain with smooth boundary, and
			let $0\le \tt <\ss < \infty$ and $1\le \pone,\ptwo <\infty$. If $\ss\in\N$, let $\pone>1$ and if $\tt\in\N_0$, let $\ptwo>1$.
			Furthermore, let
			$\ss-\tt \ge d(1/\pone -1/\ptwo)$ hold.
			Then, the embedding $\slospace{\ss}{\pone}{\Om}\hookrightarrow \slospace{\tt}{\ptwo}{\Om}$
			exists and  the following statements hold true:
			\begin{enumerate}
			\item \textbf{Sufficient requirement.}
			If $\ss-\tt >\cone+\ctwo$,
			then the embedding 2-factorizes as
			\begin{align}\label{eq:two_factorising_slobo}
				\slospace \ss \pone\Om \hookrightarrow \sobspace{\uu}{2}{\Om} \hookrightarrow  \slospace \tt \ptwo \Om
			\end{align}
			for any $\uu\in (\tt+\ctwo,\ss-\cone)$.

			\item \textbf{Necessary requirement.}
			If the embedding
			$\slospace{\ss}{\pone}{\Om}\hookrightarrow \slospace{\tt}{\ptwo}{\Om}$ is 2-factorable, then we have
			\begin{align*}
				\ss-\tt \ge \cone+\ctwo.
			\end{align*}
			\end{enumerate}
		\end{restatable}
	
		\Cref{lemma:epsapprox}, a result showing that Slobodeckij spaces are ``dense'' in the Besov-Triebel-Lizorkin spaces in a suitable way, allows analyzing 2-factorability of embeddings of Besov-Triebel-Lizorkin spaces based on \cref{lemma:slobotwofac}, yielding a central result which is our strongest answer to \eqref{eq:central_question}.

\begin{restatable}{theorem}{triebelbesovtwofac}\label{thm:two_factorisability_of_general_triebel_besov_embedds}
	Let $\Om\subset \Rd$ be a smooth bounded domain and 
	let $0\le \tt<\ss$.
	Let $\spaceone \ss\pone\qone\Om$ and $\spacetwo \tt\ptwo\qtwo \Om $ be
	spaces\footnote{Mixed cases, such as for example $\spaceone \ss\pone\qone\Om=\besovspace \ss\pone\qone\Om$ and $\spacetwo \tt\ptwo\qtwo \Om=\triebelspace \tt\ptwo\qtwo \Om$, are possible.}
	 of Triebel-Lizorkin type $\triebelspace{\ss}{p}{q}{\Om}$
	 or of Besov type $\besovspace \ss p q \Om$.
	If both spaces are of Triebel-Lizorkin type, assume $\ss-\tt\ge d/\pone-d/\ptwo$, and otherwise let $\ss-\tt> d/\pone-d/\ptwo$. 
	Then, the embedding
	\begin{align*}
		\spaceone \ss\pone\qone\Om \hookrightarrow \spacetwo \tt\ptwo\qtwo \Om 
	\end{align*}
	exists and the following statements hold true:
	\begin{enumerate}
	\item
	\textbf{Sufficient requirement.}
	If $\ss-\tt>\cone+\ctwo$, then the embedding 2-factorizes as 
	\begin{align}
		\label{eq:twofactorizinggeneralspaces}
		\spaceone \ss\pone\qone\Om\hookrightarrow \sobspace{\uu}{2}{\Om} \hookrightarrow \spacetwo \tt\ptwo\qtwo \Om
	\end{align}
	for any $\uu\in (\tt+\ctwo,\ss-\cone)$.
	If both spaces are of Triebel Lizorkin type, then the 2-factorization in \eqref{eq:twofactorizinggeneralspaces} even exists for all $\uu\in [\tt+\ctwo,\ss-\cone]\setminus\{\ss,\tt\}$. 
	
	\item
	\textbf{Necessary requirement.}
	If  $\tt>0$ and the embedding $\embe{\spaceone \ss\pone\qone\Om\ }{\spacetwo \tt\ptwo\qtwo \Om}$ is 2-factorable, then we have
	\begin{align*}
		\ss-\tt\ge\cone+\ctwo.
	\end{align*}
	\end{enumerate}
\end{restatable}

Finally, we answer Question \eqref{eq:central_question} for $F=\ell_\infty(\Om)$, characterizing of the existence of an encompassing RKHS $H$ consisting of bounded functions for  Besov-Triebel-Lizorkin spaces modulo border cases. 

\begin{theorem}\label{lemma:besov_c0_twofac}
	Let $\Om\subset \Rd$ be a smooth bounded domain and 
	let $\ss>0$. Let $\spaceone \ss\pp\qq\Om$ be a
	space 
	of Triebel-Lizorkin type $\triebelspace{\ss}{p}{q}{\Om}$
	or of Besov type $\besovspace \ss p q \Om$.
	Let $\ss> d/\pp$.
	Then, the embedding
	\begin{align*}
		\spaceone \ss\pp\qq\Om \hookrightarrow C^0(\Om)
	\end{align*}
	exists and the following statements hold true:
	\begin{enumerate}
	\item
	\textbf{Sufficient requirement.}
	If $\ss>\czero+d/2$, then the embedding 2-factorizes as 
	\begin{align*}
		\spaceone \ss\pp\qq\Om\hookrightarrow \sobspace{\uu}{2}{\Om} \hookrightarrow C^0(\Om)
	\end{align*}
	for any   $\uu\in (d/2,\ss-\czero)$.

	\item
	\textbf{Necessary requirement.}
	If the embedding $\spaceone \ss\pp\qq\Om \hookrightarrow C^0(\Om)$ is 2-factorable, then we have
	\begin{align*}
		\ss\ge\czero+d/2.
	\end{align*}
	\end{enumerate}
\end{theorem}
\cref{lemma:besov_c0_twofac} \emph{i)} shows that over a \emph{bounded} smooth domain $\Om\subseteq \Rd$, 
we find a space $E$ consisting of sufficiently smooth functions of \emph{finite smoothness} such that for $F=C^0(\Om)$ our Question \eqref{eq:central_question} allows a positive answer.
This is a strong contrast to an \emph{unbounded }domain $\Om\subseteq \Rd$, where \cref{lem:unbounded_domain} shows that \eqref{eq:central_question} is not even feasible for $E=C^\infty(\Om)\subset \ell_\infty(\Om)=F$, which stresses the necessity of the boundedness assumption.

\section{Proofs}\label{sec:proofs}

\subsection{Proofs for Section \ref{sec:general}}

For the following results recall that  
a BSF $E$ on a set $X$ is called proper if the point evaluation $\delta_x: E\to \R, f\mapsto f(x)$ is continuous for all $x\in X$, see \cref{def:bsf}.

\begin{lemma}\label{lem:inclusion-is-continuous}
 Let $F$ be a proper BSF on $X$ and $E$ be a BSF on $X$ with $E\subset F$. Then the following statements are equivalent:
 \begin{enumerate}
  \item  $E$ is proper.
  \item We have $E\hookrightarrow F$.
 \end{enumerate}
In this case we further have $\snorm{\d_x}_{E'} \leq \snorm{\id:E\to F}\cdot \snorm{\d_x}_{F'}$ for all $x\in X$.
\end{lemma}

\begin{proofof}{Lemma \ref{lem:inclusion-is-continuous}}
 \atob {ii} i We write $c\equalDef  \snorm{\id:E\to F}$ and fix an $x\in X$. Then we have $B_E \subset c B_F$, and hence we obtain
 \begin{align*}
  \snorm{\d_x}_{E'}
  = \sup_{f\in B_E} |\d_x f| \leq \sup_{f \in cB_F} |\d_x f| = c \sup_{f \in B_F} |\d_x f| = c \snorm{\d_x}_{F'}\, ,
 \end{align*}
 and in particular $\d_x:E\to \R$ is continuous.

\atob i {ii} We will apply the Closed Graph Theorem, see e.g.~\cite[Theorem 1.6.11]{Megginson98}.
To this end let us fix a sequence $(f_n)\subset E$ with
$f_n \to f$ in $E$ and $f_n\to g$ in $F$. This yields both $f_n\to f$ pointwise and $f_n\to g$ pointwise, and hence
we conclude $f=g$, that is $\id f = g$. The Closed Graph Theorem then gives the continuity of $\id:E\to F$.
\end{proofof}

\begin{proofof}{Proposition \ref{prop:closed-subspaces}}
 Let us write $\snorm f_E \equalDef \snorm f_F$ for all $f\in E$.
  Then $(E, \snorm \cdot_E)$ is complete since $E\subset F$ is closed, and
 consequently $E$ is a BSF on $X$. Moreover, we obviously have $E\hookrightarrow F$, and hence $E$ is proper by
 Lemma \ref{lem:inclusion-is-continuous}.

\atob i {ii} Since $E$, $H$, and $F$ are proper BSFs an
 application of  Lemma \ref{lem:inclusion-is-continuous}   yields $E\hookrightarrow H \hookrightarrow F$.
 Our next goal is to show that $E$ is a closed subspace of $H$. To this end,
 let us fix a sequence
$(f_n)\subset E$ and an $f\in H$ with $\snorm{f_n-f}_H\to 0$.
Since $H\hookrightarrow F$, we then have
$\snorm{f_n-f}_F\to 0$, and in particular, $(f_n)$ is a Cauchy sequence with respect to $\snorm \mycdot _F$.
However, we have $\snorm f_E = \snorm f_F$ for all $f\in E$, and therefore $(f_n)$  is also a Cauchy
sequence in $E$. Since  $(E, \snorm \cdot_E)$ is complete,
there thus exists a $g\in E$ with $\snorm{f_n-g}_E\to 0$.
Moreover, since both $H$ and $E$ are proper,
we additionally have $f_n(x)\to f(x)$ and $f_n(x)\to g(x)$ for all $x\in X$, and therefore $f=g$. This shows $f\in E$.

Finally, since $E$ is a closed subspace of $H$, we see that $(E, \snorm\cdot_H)$ is a Hilbert space, and in particular complete. In addition,
$E\hookrightarrow H$ shows that $\id: (E, \snorm\cdot_E) \to (E, \snorm\cdot_H)$ is continuous, and obviously, it is also bijective.
The Open Mapping Theorem then guarantees that the inverse is also continuous, see e.g.~\cite[Cor.~1.6.6]{Megginson98}.

\atob {ii} i If $E$ is isomorphic to a Hilbert space, we can define a Hilbert space norm on $E$ that is equivalent to $\snorm \cdot_E$.
Then $E$ equipped with this Hilbert space norm is an RKHS $H$ that satisfies $E=H$.
\end{proofof}

\begin{lemma}\label{lem:restricted-BFS}
  Let $E$ be a proper BSF on $X$ and $Y\subset X$ be non-empty. Then
  \begin{align*}
   E_{|Y} \equalDef \bigl\{ f_{|Y}: f\in E  \bigr\}
  \end{align*}
 equipped with the norm
 \begin{align*}
 \snorm g_{E_{|Y}}\equalDef \inf\bigl\{ \snorm f_E\, : \, f\in E \mbox{ and } f_{|Y} = g     \bigr\} \, , \qquad\qquad g\in E_{|Y}
 \end{align*}
 is a proper BSF on $Y$ with
 \begin{align}\label{lem:restricted-BFS-h1}
 \snorm{\d_y}_{(E_{|Y})'} \leq \snorm{\d_y}_{E'}\, , \qquad\qquad y\in Y.
 \end{align}
 Moreover,
  the restriction operator $\,\cdot\,_{|Y}: E\to E_{|Y}$ is linear and continuous with $\snorm{\,\cdot\,_{|Y}: E\to E_{|Y}}\leq 1$.
\end{lemma}

\begin{proofof}{Lemma \ref{lem:restricted-BFS}}
 Obviously, $E_{|Y}$ is a vector space consisting of functions $Y\to \R$ and the restriction operator is linear.

 To verify that $\snorm \cdot_{E_{|Y}}$
 is a norm, we first note that $\snorm \cdot_{E_{|Y}}\geq 0$ and $\snorm 0_{E_{|Y}} = 0$ are obvious.
 Moreover, if $g\in E_{|Y}$ satisfies $\snorm g_{E_{|Y}} = 0$, then there exists a sequence $(f^{(n)})\subset E$ with
 $f^{(n)}_{|Y} = g$ and $\snorm {f^{(n)}} \to 0$. The latter gives $f^{(n)}(x) \to 0$ for all $x\in X$, and hence we find
 \begin{align*}
  g(y) = f^{(n)}_{|Y}(y) = f^{(n)}(y) \to 0 \, , \qquad \qquad y\in Y.
 \end{align*}
This shows $g= 0$. In addition, the homogeneity follows from
\begin{align*}
 \snorm {\a g}_{E_{|Y}}
= \inf\bigl\{ \snorm f_E\, : \, f\in E \mbox{ and } f_{|Y} = \a g     \bigr\}
= \inf\bigl\{ \snorm {\a f}_E\, : \, f\in E \mbox{ and } f_{|Y} =   g     \bigr\}
= \a\snorm { g}_{E_{|Y}}\, .
\end{align*}
To verify the triangle inequality, we fix some $g_1, g_2\in E_{|Y}$. Since for all $f^{(1)},f^{(2)}\in E$ with
$f^{(i)}_{|Y} = g_i$ we have $(f^{(1)} + f^{(2)})_{|Y} = g_1+g_2$ we then obtain
\begin{align*}
 \snorm {g_1+g_2}_{E_{|Y}}
 &\leq  \inf\bigl\{ \snorm {f^{(1)}+f^{(2)}}_E\, : \, f^{(i)}\in E \mbox{ and }f^{(i)}_{|Y} = g_i \bigr\} \\
 &\leq  \inf\bigl\{ \snorm {f^{(1)}}_E + \snorm {f^{(2)}}_E\, : \, f^{(i)}\in E \mbox{ and }f^{(i)}_{|Y} = g_i \bigr\} \\
 &= \snorm {g_1}_{E_{|Y}} +  \snorm {g_2}_{E_{|Y}} \, .
\end{align*}

 Let us now show that restriction operator  $\,\cdot\,_{|Y}: E\to E_{|Y}$ is continuous. To this end, we fix an $f\in E$. Then we easily find
\begin{align*}
 \snorm{f_{|Y}}_{E_{|Y}} = \inf\bigl\{ \snorm {\tilde f}_E\, : \, \tilde f\in E \mbox{ and } \tilde f_{|Y} = f_{|Y}     \bigr\}\, ,
 \leq \snorm f_E\, .
\end{align*}
and hence we have the desired continuity with  $\snorm{\,\cdot\,_{|Y}: E\to E_{|Y}}\leq 1$.

 Next, we show that the norm is complete. To this end, let $(g_n)\subset E_{|Y}$
be a sequence with
\begin{align*}
 \sum_{n=1}^\infty \snorm{g_n}_{E_{|Y}} < \infty.
\end{align*}
By a well-known characterization of \emph{complete} normed spaces, see e.g.~\cite[Thm.~1.3.9]{Megginson98},
it suffices to show that $\sum_{n=1}^\infty g_n$ converges in $E_{|Y}$.
To this end, we choose $f^{(n)}\in E$ with  $f^{(n)}_{|Y} = g_n$ and $\snorm{f^{(n)}}_E \leq \snorm{g_n}_{E_{|Y}} + n^{-2}$ for all $n\geq 1$.
This gives
\begin{align*}
 \sum_{n=1}^\infty \snorm{f^{(n)}_{|Y}}_E < \infty,
\end{align*}
and since $E$ is complete, we see with the help of \cite[Thm.~1.3.9]{Megginson98}
that $f \equalDef \sum_{n=1}^\infty f^{(n)}$ exists with convergence in $E$.
Using the already established continuity of the restriction operator  $\,\cdot\,_{|Y}: E\to E_{|Y}$
we conclude that
\begin{align*}
 \biggl( \sum_{n=1}^\infty f^{(n)} \biggr)_{|Y}
 =
  \biggl( \lim_{m\to \infty} \sum_{n=1}^m f^{(n)} \biggr)_{|Y}
  =
  \lim_{m\to \infty} \biggl( \sum_{n=1}^m f^{(n)} \biggr)_{|Y}
  = \lim_{m\to \infty}  \sum_{n=1}^m f^{(n)}_{|Y}
  =\sum_{n=1}^\infty g_n\, ,
\end{align*}
where the convergence in the last three expressions take place in $E_{|Y}$ by the continuity of the restriction operator.

Finally, to check that $E_{|Y}$ is proper, we fix an $y\in Y$.
For $g\in E_{|Y}$ and an $f\in E$ with $f_{|Y}= g$ we then obtain
\begin{align*}
 |\d_y g| = |g(y)| = |f_{|Y}(y)| = |\d_y f|   \leq \snorm{\d_y}_{E'} \cdot    \snorm f_E  \, ,
\end{align*}
and by taking the infimum over all $f\in E$ with $f_{|Y}= g$ we then see that $|\d_y g| \leq \snorm{\d_y}_{E'} \cdot  \snorm g_{E_{|Y}}$.
The latter also shows \eqref{lem:restricted-BFS-h1}.
\end{proofof}

\begin{proofof}{Proposition \ref{prop:restricted-inclusions}}
 We first show that $E_{|Y} \subset  F_{|Y}$.
 To this end, we fix some $g\in E_{|Y}$. Then there exists an $f\in E$ with $f_{|Y} = g$. We then know that $f\in F$, which in
 turn implies $g = f_{|Y} \in  F_{|Y}$.
 Moreover, we find
 \begin{align*}
  \snorm g_{F_{|Y}}
  =  \inf\bigl\{ \snorm {\tilde f}_F\, : \, \tilde f\in F \mbox{ and } \tilde f_{|Y} = g     \bigr\}
  \leq \snorm f_F
  \leq \snorm{\id:E\to F} \cdot \snorm f_E\, ,
 \end{align*}
and by taking the infimum over all $f\in E$ with $f_{|Y} = g$ we then find $\snorm g_{F_{|Y}}  \leq  \snorm{\id:E\to F} \cdot  \snorm g_{E_{|Y}}$.
\end{proofof}

\begin{proofof}{Lemma \ref{lem:2fact-implies-twotype}}
Let us fix a 2-factorization \eqref{eq:twofactorisation}. Then we can expand this to
		\begin{center}
			\begin{tikzcd}
				E \arrow[d,"U"] \arrow[r, "\id"]&   F \\
				H \arrow[r,"\id_H"] & H\arrow[u,"V"]
			\end{tikzcd}
		\end{center}
Since $H$ is of type 2, so is $\id_H$, and therefore \eqref{eq:type-2-ideal-norm} shows that $A$
is of type 2, too. Analogously, we find that $A$  is of cotype 2.
\end{proofof}

\begin{proofof}{\cref{theorem:2-fac_in_bfs}}
	If we have an RKHS $H$ with $E\subset H\subset F$, then Lemma \ref{lem:inclusion-is-continuous}
	shows $E\hookrightarrow H_0\hookrightarrow F$, and this in turn directly provides a 2-factorization.
	To show the converse, assume a 2-factorization
	\begin{center}
		\begin{tikzcd}
			E \arrow[rd,"U"] \arrow[hookrightarrow,rr, "\id"]&  & F \\
			& H_0 \arrow[ru,"V"] &
		\end{tikzcd}.
	\end{center}
	We define the feature map $\phi:X\to H_0$ by
	$$
	\phi(x)  \equalDef  V^* \d_{x,F}\,  ,
	$$
	where $\d_{x,F}\in F'$ denotes the evaluation functional at $x$ acting on $F$ and
	$\adjoint V:\dual F\to { H_0}$ is the adjoint operator
	of $V$ that is uniquely determined by
	\begin{align*}
	\scal  {V^*f'}{h}_{H_0}  = \scal {f'}{Vh}_{F',F}\, , \qquad \qquad h\in H_0, f'\in F'.
	\end{align*}
	The corresponding kernel $k:X\times X\to \R$ is
	\begin{align*}
		k(x,x') \equalDef \scal{ \phi(x)}{\phi(x')}_{{H_0}}
	\end{align*}
	and by \cite[Thm.~4.21]{StCh08} its RKHS is given by
	\begin{align}\label{eq:H-descript}
	H
	=\left\{ \scal{h}{\phi(\cdot)}_{{H_0}} \, \middle \vert\, {h}\in {H_0}\right\} \, .
	\end{align}
	Our next goal is to show  $H=\ran V$.
	To this end, we first observe that for $h\in H_0$ and $x\in X$ we have
	\begin{align*}
		\scal{{h}}{\phi(x)}_{H_0}
		&= \scal{{h}}{V^* \d_{x,F}}_{H_0}
		=\scal{V^* \d_{x,F}}{h}_{H_0}
		= \scal{\d_{x,F}}{V {h}}_{F',F}
		= (V{h})(x) \, .
	\end{align*}
	Now, if choose an $f\in H$, then by \eqref{eq:H-descript} there exists an $h\in H_0$
	with $f= \scal{{h}}{\phi(\cdot)}_{{H_0}}$,
	and hence our observation yields $f = Vh\in \ran V$.
	Conversely, if we fix an $f\in \ran V$, then there exists an
	$h\in H_0$
	with $Vh =f$, and our observation gives $f= \scal{{h}}{\phi(\cdot)}_{H_0}\in H$
	by \eqref{eq:H-descript}.
	We  conclude that $H=\ran V$, which in particular implies $H\subset F$.

	To establish $E\subset H$, we simply observe that the 2-factorization together with $H=\ran V$
	yields $E	= VUE\subset \ran V= H$.
\end{proofof}

\begin{proofof}{Theorem \ref{thm:suff-2-fact}}
\ada i By Pietsch's factorization theorem, see e.g.~\cite[Cor.~2.16]{DiJaTo95}, we know that $\id: E\to F$ is 2-factorable,
and hence the assertion follows by Theorem \ref{theorem:2-fac_in_bfs}.

\ada {ii} By Kwapien's theorem, see e.g.~\cite[Thm.~12.19]{DiJaTo95}, we know that $\id: E\to F$ is 2-factorable,
and therefore the assertion again follows by Theorem \ref{theorem:2-fac_in_bfs}.
\end{proofof}

\begin{proofof}{Corollary \ref{cor:nec-2-fact}}
By Theorem \ref{theorem:2-fac_in_bfs}  we know that $\id:E\to F$ is 2-factorable, and hence the assertion follows by Lemma \ref{lem:2fact-implies-twotype}.
\end{proofof}

\subsection{Proofs for \cref{sec:examples}}

\subsubsection{Proofs for \cref{subsection:abstact_holder}}

\donee
Before we present the proofs of \cref{subsection:abstact_holder}, we recall that given a metric space $(X,d)$
and some $\a\in (0,1]$, the map $d^\a:X\times X\to [0,\infty)$ is another metric on $X$, which generates the same topology as $d$ does.
Moreover, an $f:X\to \R$ is $\a$-H\"older continuous with respect to $d$, if and only if $f$ is 1-H\"older continuous
with respect to $d^\a$, and we have
\begin{align}
	\label{eq:holder_changing_metric}\mholnorm{f}{\al}{\met}=\mholnorm{f}{1}{\met^\al} \, .
\end{align}
Finally, an elementary calculation shows $\pac {\d} {\met^\al} X = \pac {\d^{1/\a}} {\met} X$ for all $\d>0$.

We construct suitable functions with disjoint support, for which
in this section a cotype 2 argument yields the results.
To this end, consider, for fixed center $t\in X$, radius $\d>0$, and $\al \in (0,1]$ the bump function $f_{t,\d,\a}:X\to\R$ defined by
	\begin{align}\label{eq:hoelder-bump-function}
		f_{t,\d,\a}(s)\equalDef
		\pospart{\d-\met^\al(t,s)}\, , \qquad \qquad s\in X.
	\end{align}
Now, the basic idea is to consider centers $t_1,\dots,t_n$ and a radius $\d$, for which the corresponding
functions $f_{t_i,\d,\a}$ have disjoint support, since in this case the Rademacher norms can be
suitably controlled as the following result shows. 

\begin{proposition}\label{prop:metric_bumps_properties}
	Let $(X,\met)$ be a metric space and  $0<\be\le\al\le 1$. Moreover, let $t_1,\dots,t_n\in X$ and $0<\d\le 1$. Then the following statements hold
	true for the bump functions $f_{t_i,\d,\a}$ defined in \eqref{eq:hoelder-bump-function}:
	\begin{enumerate}
	\item We have $\supp f_{t_i,\d,\a} \subset \ball{\d^{1/\a}}{t_i}$ and $\inorm{f_{t_i,\d,\a}} = \d$.
	\item If   $\met^\al(t_i,t_j)\ge 3\d$ holds for all $i\ne j$, then, for all $\varepsilon_1,\dots,\varepsilon_n \in \{-1,1\}$, we have 
	\begin{align*}
	\Lmholnorm{\sum_{i=1}^n \varepsilon_i \metbump {t_i} {\d}{\al}}{\al}{\met} \le  1 \, .
	\end{align*}
	 \item If there exist  $s_1,\dots,s_n\in X$ with $\met(s_i,t_i)=\d^{1/\al}$, then
		$\mholnorm{ \metbump {t_i} {\d}{\met^\al}}{\be}{\met} \ge \d^{(\al-\be)/\al}$ for all $i=1,\dots,n$.
	\end{enumerate}
\end{proposition}

\begin{proof}[Proof of \cref{prop:metric_bumps_properties}]
\ada i The construction of $f_{t_i,\d,\a}$ gives $\{  f_{t_i,\d,\a} \neq 0 \} \subset \ball{\d^{1/\a}}{t_i}$, and since the latter ball is closed, we obtain the first assertion. The second assertion is trivial.

\ada {ii} To emphasize the role of the metric in the definition of the bump function, we write $f_{t_i,\d,\a,\met} \equalDef f_{t_i,\d,\a}$ for a moment.
By switching to the metric $\met^\al$ we then have $f_{t_i,\d,\a,\met} = f_{t_i,\d,1,\met^\a}$.
Consequently, \cref{eq:holder_changing_metric}
shows
\begin{align*}
\Lmholnorm{\sum_{i=1}^n \varepsilon_i \metbump {t_i} {\d}{\al}}{\al}{\met}
= \Lmholnorm{\sum_{i=1}^n \varepsilon_i \metbump {t_i} {\d}{\a,\met}}{\al}{\met}
= \Lmholnorm{\sum_{i=1}^n \varepsilon_i \metbump {t_i} {\d}{1, \met^\al}}{1}{\met^\al} \, .
\end{align*}
Without loss of generality, it thus suffices to investigate the case $\a=1$ for
\begin{align*}
\fun \equalDef \sum_{i=1}^n \varepsilon_i \metbump {t_i} {\d}{1} \, .
\end {align*}
Here we first note that $\emph{i)}$ together with our assumption on $t_1,\dots,t_n$ implies
$\met(\supp f_{t_i,\d,1} ,\supp f_{t_j,\d,1} ) \geq \d$
for all $i\neq j$. This gives $\inorm f\leq \d \leq 1$ by $\emph{i)}$.
It remains to show
\begin{align}\label{eq:holder_goal}
	\abs{f(x)-f(y)} \leq \met(x,y) \, , \qquad\qquad x,y\in X \, .
\end{align}
To this end, we define the ``center identifying function''   $\iota:X\to \setrange 0 n$
by
\begin{align*}
	\iota(x)=
	\begin{cases}
		i\, , &\text{ if there exists an }i\in \setrange 1n \text{ such that }d(x,t_i)<\d\\
		0\, ,  &\text{ if }d(x,t_i)\ge\d \text{ for all }i\in \setrange 1n.
	\end{cases}
\end{align*}
With the help of this notation we can express $f$ by
\begin{align*}
f(x) =
	\begin{cases}
		\e_{\iota(x)}    \bigl(\d -\met(t_{\iota(x)},x)\bigr)\, ,  &\text{ if }\iota(x) > 0,\\
		0\, , &\text{ if }\iota(x) = 0.
	\end{cases}
\end{align*}
In the following we thus investigate \eqref{eq:holder_goal} for the possible values of $\iota(x), \iota(y)$.

\textbf{Case $\iota(x)=\iota(y) =0$.} Trivial.

\textbf{Case $\iota(x)=\iota(y) >0$.}
Here we have  $\abs{f(x)-f(y)}= \abs{\met(t_{\iota(x)},x) -\met(t_{\iota(x)},y)}\le \met(x,y)$.

\textbf{Case $\iota(y)\neq \iota(x)$ and $\iota(x)=0$.} Here, $\iota(x)=0$ implies
$\d \leq \met (x,t_{\iota(y)}) \leq  \met(x,y)+ \met(y,t_{\iota(y)})$, and hence we find
$\abs{f(x)-f(y)}= \d - \met(y,t_{\iota(y)}) \leq \met(x,y)$ as desired.

\textbf{Case $\iota(x)\ne\iota(y)$ and $\iota(x),\iota(y)>  0$.} Here we first note that  $\met(x,y)\ge \d$.
If $\varepsilon_{\iota(x)} = \varepsilon_{\iota(y)}$, we thus find
\begin{align*}
\abs{f(x)-f(y)}
= \abs{ \met(y,t_{\iota(y)}) -\met(x,t_{\iota(x)}) 	}
\le \max\{\met(y,t_{\iota(y)}) , \met(x,t_{\iota(x)})  \}
\le \d
\le \met(x,y)\, .
\end{align*}
Moreover, if  $\varepsilon_{\iota(x)} \neq \varepsilon_{\iota(y)}$,  we obtain
\begin{align*}
\abs{f(x)-f(y)}
=
\d - \met(x,t_{\iota(x)}) + \d - \met(y,t_{\iota(y)})
\leq
3\d - \met(x,t_{\iota(x)}) - \met(y,t_{\iota(y)})
\leq \met(x,y)\, ,
\end{align*}
where in the last step we  used $3\d \leq \met(t_{\iota(x)},t_{\iota(y)}) \leq  \met(x,t_{\iota(x)}) + \met(x,y) +\met(y,t_{\iota(y)})$.

\ada {iii} For simplicity, we write $s\equalDef s_i$ and $t\equalDef t_i$. Then we have
		\begin{align*}
			&\mholnorm{\metbump{t}{\d}{\al}}{\be}{\met}
			=
			\mholnorm{\metbump{t}{\d}{\al}}{1}{\met^\be}
			\ge \frac{\abs{\metbump{t}{\d}{ \al}(t)-\metbump{t}{\d}{ \al}(s)}}{\met^\be(s,t)}=
			\frac{\abs{\d - (\d-\met^\al(s,t))}}{\met^\be(s,t)}=
			\frac{\d}{\d^{\be/\al}}
			= \d^{(\al-\be)/\al}\, ,
		\end{align*}
		where in the first step we used \eqref{eq:holder_changing_metric}.
\end{proof}

\donee

\begin{lemma}\label{lem:connect-lemma}
Let $(X,d)$ be a connected metric space, $\d_0>0$, and $t\in X$. If there
exists an $s_0\in X$ with $\met(s_0,t) >  \d_0$, then for all $\d\in (0,\d_0]$ there
there   exists an $s\in X$ with $\met(s,t) = \d$.
\end{lemma}

\begin{proof}[Proof of \cref{lem:connect-lemma}]
Let $\d\in (0,\d_0)$ and assume that for all $s\in X$ we have $\met(s,t) \neq \d$. Then we obtain a partition
\begin{align*}
X=\{ s\in X\mid \met(s,t)<\d \} \cup  \{s\in X\mid \met(s,t) > \d \}
\end{align*}
of $X$ into two non-empty, open sets. This contradicts the connectedness assumption.
\end{proof}

\begin{proof}[Proof of \cref{thm:abstract_holder_twofac_thm}]
	If $X$ has only one element, the claim is trivial. Hence, we assume that $X$ has at least two elements, which in turn implies
	$\diam X >0$. We define $\d_0 \equalDef \min\{1, (\diam X)^{\al}  \} /3$.

	For $0<\d<\d_0$ there then exist $s,t\in X$ with $\met^\a(s,t)   > 3\d$, and hence we obtain
	$n\equalDef \pac {3\d} {\met^\al} X \geq 2$. Let us fix a $3\d$-packing $t_1,\dots,t_n\in X$.
	For a fixed $i\in \setrange 1n$ and all $j\in \setrange 1n$ with $j\neq i$ we then have
	$\met^\al(t_j,t_i)\geq 3\d > \d$, where we note that $n\geq 2$ ensures that there actually exists such a $j\neq i$.
	Consequently,
	 \cref{lem:connect-lemma} gives an $s_i\in X$ with $\met^\a(s_i,t_i) = \d$.
	For $1\le i\le n$, we write
	$f_{i}\equalDef \metbump{t_i}{\d}{\al}$, where $\metbump{t_i}{\d}{\al}$ is the bump function defined in \eqref{eq:hoelder-bump-function}.
	Then, by \cref{prop:metric_bumps_properties} we have
	\begin{align*}
		\radnorm{(f_1,\dots,f_n)}{n}{\mholspace X {\al}\met }
		&= \E_{\radsequence{\varepsilon}}
			\norm[\Big]{	\sum_{i=1}^n \varepsilon_i f_i}_{\hoel_{\al,\met}}
		\le 1
		\intertext{and}
		\elltwonorm{(f_1,\dots,f_n)}{n}{\mholspace X {\be }\met   } &= \bra[\Big]{\sum_{i=1}^n \mholnorm{ f_i}\be \met ^2 }^{1/2} \ge n^{1/2}  \d^{(\al-\be)/\al}\, .
	\end{align*}
	Combining both estimates yields
	\begin{align*}
		C\equalDef \precotypenorm{\embe{\mholspace X \al \met}{\mholspace X \be \met}}{2}{n}
		\ge \frac{\elltwonorm{(f_1,\dots,f_n)}{n}{\mholspace X {\be }\met   }  }{\radnorm{(f_1,\dots,f_n)}{n}{\mholspace X {\al}\met } } \ge n^{1/2} \d^{(\al-\be)/\al}.
	\end{align*}
	Now \cref{cor:nec-2-fact} shows $C<\infty$, and using the definition of $n$ we get
	\begin{align*}
	\pac {(3\d)^{1/\a}} {\met} X = \pac {3\d} {\met^\al} X = n \leq C^2 \d^{2(\al-\be)/\al}\, ,
	\end{align*}
 where the first identity was already mentioned around \eqref{eq:holder_changing_metric}. A simple variable transformation then yields
 the assertion for all $\d < \min\{1, \diam X\}$.

To establish the second assertion, we note that even without the connectivity assumption
\begin{align*}
	\elltwonorm{(f_1,\dots,f_n)}{n}{\ell_\infty(X)}
	&= \bra[\Big]{  \sum_{i=1}^n \norm{ f_i}_{\ell_\infty(X)}^2 }^{1/2} = n^{1/2} \d
\end{align*}
holds.
Repeating the cotype 2 argument used above with $\b=0$ then yields the second assertion.
\end{proof}

\donee

\begin{proof}[Proof of \cref{lem:unbounded_domain}]
	Since $\Om$ is unbounded  we have $\paceuclid{\d}{\Om}=\infty$ for all $\d>0$:
	Indeed, if we had $\paceuclid{\d}{\Om}<\infty$, then there would be a maximal, finite $\d$-packing $x_1,\dots,x_n\in \Om$.
	For every
	$x\in \Om$ there would thus exist an $x_i$ with $\met(x,x_i)< \d$
	and $\Om\subset \ball\d{x_1}\cup\dots\cup \ball\d{x_n}$ would be bounded, which contradicts our assumption.
	We now set $\d\equalDef 3$  and choose  an infinite $\d$-packing $x_1,\dots\in \Om$.
	Moreover, we consider the bump function $\fun\in C^\infty(\Rd)$ from \cref{eq:bump_function} and define $\fun_i:\Om\to\R$ by $\fun_i(x)\equalDef \fun(x-x_i)$. For every $n\in\N$ we observe
	\begin{align*}
	\radnorm{(\fun_1,\dots,\fun_n)}{n}{C^\infty(\Om)}\le  \norm {\fun }_{C^\infty(\Rd)} < \infty
	\end{align*}
	as the supports of $\fun_1,\dots,\fun_n$ are disjoint by construction and $\fun \in \calS(\Rd)$.
	Furthermore, we have $\inorm{f_i} = f(0) = 1$ for all $i=1,\dots,n$, and hence we obtain
	$\elltwonorm{(\fun_1,\dots,\fun_n)}{n}{\ell_{\infty}({\Om})}=n^{1/2}$.
	We observe
	\begin{align*}
		\cotypenorm{\embe {C^\infty(\Om)}{\ell_\infty( \Om)}} \ge \frac{ \elltwonorm{(\fun_1,\dots,\fun_n)}{n}{\ell_{\infty}({\Om})}}{\radnorm{(\fun_1,\dots,\fun_n)}{n}{C^\infty(\Om)}} \ge \frac{n^{1/2}}{\norm{\fun}_{C^{\infty}(\Rd)}}
	\end{align*}
	for all $n\geq 1$. This implies $\cotypenorm{\embe {C^\infty(\Om)}{\ell_\infty( \Om)}} = \infty $ and by \cref{cor:nec-2-fact},
	there cannot exists an RKHS $H$ with $C^\infty(\Om)\subset H\subset \ell_{\infty}({\Om})$, that is, there is no
	RKHS $H$ with bounded kernel and $C^\infty(\Om)\subset H$.
\end{proof}

\subsubsection{Proofs for \cref{subsubsection:sobolev_spaces_of_mixed_smoothness}}

\donee

The proofs of \cref{subsubsection:sobolev_spaces_of_mixed_smoothness,subsubsection:isotropic_spaces}are based on the same bump function $\fun:\Rd\to\R$ given by
\begin{align}
	\label{eq:bump_function}
	\fun(x)\equalDef 
	\begin{cases}
		\frac{1}{e} \exp\bra{-\frac{1}{1-\norm x^2}}, & \mbox{ if } x\in  \ball {1}{0},\\
		0,& \text{ otherwise.}
	\end{cases}
\end{align}
Obviously, we have $\supp(\fun)\subset \ball 1 0$. Moreover, we have
$\fun\in \calS(\Rd)$, that is, $\fun$ is a Schwartz function \cite{AdFo03}.
Finally, for all multi-indices $\al\in\N_0^d$ it holds  $\partial_\al \fun\ne 0$.

\donee

\begin{lemma}\label{lemma:asymptotics_in_wsp_of_rescaled_bumps_simple_argument}
	Let $\Omega\subset \R^d$ be open such that $\ball 10\subset \Om$. Let $\al\in\N_0^d$, and $\pp\in [1,\infty)$.
	Moreover, for $\d \in (0,1/2]$ and suitable  $n\ge 1$   let $x_1,\dots,x_n\in \ball {1/2}0$ be a $3\d$-packing.
	We define $\fun_i:\Om\to\R$ by $ \fun_i(x)\equalDef \fun(\d^{-1}x -\d^{-1}x_i)$,
	where $\fun:\R^d\to\R $ is the bump function from \cref{eq:bump_function}.
	Moreover, 
	for  $\varepsilon_1,\dots,\varepsilon_n\in\set{-1,1}$ we define $\hfun:\Om\to\R$ by
	\begin{align*}
		\hfun(x)\equalDef \sum_{i=1}^n \varepsilon_i\fun_i(x)\, .
	\end{align*}
	Then the functions $\fun_1,\dots,\fun_n$ have disjoint support and the function $h$ fulfills
	\begin{align*}
		\lpnorm {\partial_\al \hfun}{\pp}\Om  =n^{1/p}\d^{d/\pp-\munorm\al} \lpnorm {\partial_\al \fun}{\pp}\Om  \, .
	\end{align*}
\end{lemma}

\begin{proof}
	We define $z_i\equalDef \d^{-1}x_i$. 
	For $x\in \ball 10$ we then have
	$\snorm {x+z_i} \leq \snorm x + \snorm {z_i} \leq 1 + \d^{-1}/2 \leq \d^{-1}$, that is $\ball 10 \subset \ball{\d^{-1}}{-z_i}$.
	Moreover, we observe the identity
	\begin{align*}
		\bra{\partial_\al \hfun}(x)= \negpow\d{\munorm\al } \isum \varepsilon_i\partial_\al\fun\bra{\negpow \d 1x-z_i}\, , \qquad \qquad x\in \Om.
	\end{align*}
	We define $D_i\equalDef \supp\bra{x\mapsto \partial_\al \fun\bra{\negpow \d 1x-z_i}}$. By $\supp(\partial_\al \fun)\subset \supp(\fun)\subset \ball 10$ and  our construction 
	we have $D_i\subset \ball{\d}{x_i}$, and especially the sets $D_1,\dots,D_n$ are pairwise disjoint.
	This directly yields
	\begin{align*}
		\norm{\partial_\al \hfun}_{\lpspace \pp \Om}^\pp 
		&=\d^{-p \munorm\al}\int_\Om \abs[\Big]{\isum \varepsilon_i\partial_\al\fun\bra{\negpow \d 1x-z_i}}^p \dx x
		=\isum\d^{-p \munorm\al}\int_\Om 
		\abs[\Big]{ \varepsilon_i\partial_\al\fun\bra{\negpow \d 1x-z_i}}^p \dx x \, .
	\end{align*}
	Moreover, using the substitution $y\equalDef \d^{-1}x -z_i$, that is $x= \d(y+z_i) = \d y + x_i$,  we find
	\begin{align*}
	\int_\Om\abs[\Big]{ \varepsilon_i\partial_\al\fun\bra{\negpow \d 1x-z_i}}^p \dx x
	=
	\d^d\int_{\Rd} \eins_\Om(\d y + x_i) \abs {\partial_\al\fun\bra{y}}^p \dx y
	&=
	\d^d\int_{-z_i + \d^{-1} \Om }  \abs {\partial_\al\fun\bra{y}}^p \dx y \\
	&=
	\d^d\int_{\ball 10}  \abs {\partial_\al\fun\bra{y}}^p \dx y \\
	&=
	\d^d \int_{\Om}  \abs {\partial_\al\fun\bra{y}}^p \dx y \, ,
\end{align*}
where in the last two steps we used the inclusions $\supp(\partial_\al\fun) \subset \supp(\fun) \subset \ball 10$ and
$\ball 10 \subset \ball{\d^{-1}}{-z_i} \subset -z_i + \d^{-1} \Om$ as well as $\ball 10\subset \Om$.
Combining both calculations then yields the assertion.
\end{proof}

\donee

\begin{lemma}\label{lem:radnorm-mixed-sobol}
	Let $\Omega\subset \R^d$ be open such that $\ball 10\subset\Om$ and let $\pp\in [1,\infty)$.
	Moreover, for $\d \in (0,1/2]$ and suitable $n\ge 1$  let $x_1,\dots,x_n\in \ball {1/2}0$ be a $3\d$-packing, and
	let $\fun_i:\Om\to\R$ be defined  by
	\begin{align*}
		\fun_i(x)\equalDef \fun(\d^{-1}x-\d^{-1}x_i)\, ,
	\end{align*}
	where   $\fun$ is the bump function from \cref{eq:bump_function}. Then for all coherent $A\subset \N_0^d$ and $\a_0\in A$ with $s\equalDef|\a_0|_1 = |A|_1$
	we have
	\begin{align}
		\label{eq:mixed_sob_radnorm}
		\lpnorm{\partial_{\a_0} \fun}p\Om \cdot {n}^{1/p}\d^{d/p - s}
		&\leq
		\radnorm{(\fun_1,\dots,\fun_{n})}{{n}}{\slospace\A p\Om}
		\leq
		\sup_{\a\in A}  \lpnorm{\partial_\a \fun}p\Om    \cdot  {n}^{1/p}\d^{d/p - s}\, ,
		\intertext{and }
		\lpnorm{\partial_{\a_0} \fun}p\Om \cdot {n}^{1/2}\d^{d/p - s}
		&\leq
		\elltwonorm{(\fun_1,\dots,\fun_{n})}{{n}}{\slospace\A p\Om}
		\leq
		\sup_{\a\in A}  \lpnorm{\partial_\a \fun}p\Om    \cdot  {n}^{1/2}\d^{d/p - s}\, .
		\nonumber
		\end{align}
\end{lemma}

\begin{proof}
From \cref{lemma:asymptotics_in_wsp_of_rescaled_bumps_simple_argument} we directly obtain
	\begin{align*}
		\radnorm{(\fun_1,\dots,\fun_{n})}{{n}}{\slospace\A p\Om}
		= \E_{\radsequence{\varepsilon}} \slonorm{\varepsilon_1 \fun_1+\dots+\varepsilon \fun_{n}}{\A}{p}{\Om}
		= \sup_{\a\in \A} {n}^{1/p}\d^{d/p - |\a|_1} \cdot \lpnorm{\partial_\a \fun}p\Om\, .
	\end{align*}
	Now the lower bound on $\radnorm{(\fun_1,\dots,\fun_{n})}{{n}}{\slospace\A p\Om}$ is obvious, and the upper bound quickly follows by
	remembering $\d\leq 1/2$.
	
Moreover, applying \cref{lemma:asymptotics_in_wsp_of_rescaled_bumps_simple_argument} in the case $n=1$ shows
$\lpnorm {\partial_\al \fun_i}{\pp}\Om  = \d^{d/\pp-\munorm\al} \lpnorm {\partial_\al \fun}{\pp}\Om$ for all $i=1,\dots,n$, which in turn implies
\begin{align*}
\snorm{f_i}_{\slospace\A p\Om} = \sup_{\a\in A} \d^{d/\pp-\munorm\al} \lpnorm {\partial_\al \fun}{\pp}\Om \, .
\end{align*}
Hence we obtain
\begin{align*}
\elltwonorm{(\fun_1,\dots,\fun_{n})}{{n}}{\slospace\A p\Om}
=
n^{1/2} \sup_{\a\in \A}  \d^{d/\pp-\munorm\al} \lpnorm {\partial_\al \fun}{\pp}\Om\, .
\end{align*}
The rest of the proof is analogous to the proof of \eqref{eq:mixed_sob_radnorm}.
\end{proof}

\donee

\begin{proof}[Proof of \cref{thm:mixed_twofac}]
	By a simple translation and scaling argument we may assume without loss of generality that $\ball 10\subset \Om$. 
	Let us fix a  $0<\d\le 1/2$. We write $n\equalDef \paceuclid{3\d}{\ball {1/2}0}$ and choose a $3\d$-packing $x_1,\dots,x_{n}\in \ball {1/2}0$. Moreover, we define
	$\fun_1,\dots,\fun_n$ as in \cref{lem:radnorm-mixed-sobol}.
	By \cref{lem:radnorm-mixed-sobol} there then exists a constant $C$ that is independent on $n$ and $\d$ such that
	\begin{align*}
		\typenorm{\embe{\slospace \A\pone\Om}{\slospace\B\ptwo\Om}}
		\ge
		\frac{
			\radnorm{(\fun_1,\dots,\fun_{n})}{{n}}{\slospace\B\ptwo\Om}}
		{\elltwonorm{(\fun_1,\dots,\fun_{n})}{{n}}{\slospace\A\pone\Om}}
		\ge
		 C {n}^{1/\ptwo-1/2} \d^{\ss-\tt + d(1/\ptwo-1/\pone)} \, .
	\end{align*}
	Using the packing number bound \eqref{eq:cov-num-Rd} we thus find some constant $\widehat C>0$ such that 
	\begin{align*}
		\typenorm{\embe{\slospace \A\pone\Om}{\slospace\B\ptwo\Om}}
		\ge \widehat C \d^{\ss-\tt +d(1/2-1/\pone)}\, .
	\end{align*}
	holds for all $\d\in (0,1/2]$.
	Now \cref{lem:2fact-implies-twotype} implies that 
	$\ss-\tt\ge \pospart{d/\pone -d/2}$ holds. 
	Using a cotype 2 argument, we analogously obtain the requirement $\ss-\tt\ge \pospart{d/2 -d/\ptwo}$.

	It remains to show that  $\ss-\tt\ge \cone+\ctwo$. However, if  $\cone=0$ or $\ctwo=0$, then the claim immediately follows from our previous
	considerations.
	Moreover, if both are positive, we have
	\begin{align*}
	\cone+\ctwo=d(1/\pone-1/\ptwo)\, ,
	\end{align*}
	 and the claim holds by assumption. 
\end{proof}

\donee
\begin{proof}[Proof of \cref{thm:mixed_twofac_C-0}]
	Recall that if there exists an   RKHS $H$ with 
	 ${\slospace\A\pp\Om} \subset H\subset {\ell_\infty(\Om)}$, then $\embe{\slospace\A\pp\Om}  {\ell_\infty(\Om)}$ is 2-factorable,
	 and therefore of type 2 and cotype 2.
	 
	By a simple translation and scaling argument we may assume without loss of generality that $\ball 10\subset \Om$. 
	Let us fix a  $0<\d\le 1/2$. We write $n\equalDef \paceuclid{3\d}{\ball {1/2}0}$ and choose a $3\d$-packing $x_1,\dots,x_{n}\in \ball {1/2}0$.
	For the functions $f_1,\dots,f_n$ considered in \cref{lemma:asymptotics_in_wsp_of_rescaled_bumps_simple_argument} we then have
	\begin{align*}
		\radnorm{(\fun_1,\dots,\fun_{n})}{{n}}{\ell_\infty(\Om)}=\E_{\radsequence{\varepsilon}} \norm{\varepsilon_1 \fun_1 +\dots +\varepsilon_n \fun_n}_{\ell_\infty(\Om)}= \supnorm{\fun_1}=1,
	\end{align*}
	since the functions $\fun_1,\dots,\fun_n$ are disjointedly supported translated copies of each other. \cref{lemma:asymptotics_in_wsp_of_rescaled_bumps_simple_argument,lem:radnorm-mixed-sobol} yield
	\begin{align*}
		\typenorm{ \embe{\slospace\A\pp\Om}  {\ell_\infty(\Om)}} 
		\ge 
		\frac
		{\radnorm{(\fun_1,\dots,\fun_{n})}{{n}}{\ell_\infty(\Om)} }
		{\elltwonorm{(\fun_1,\dots,\fun_{n})}{{n}}{\slospace\A\pp\Om}}
		\ge C {n}^{-1/2} \d^{\ss - d/\pp}
	\end{align*}
	for some suitable constant $C>0$, and using the packing number bound \eqref{eq:cov-num-Rd}  \cref{lem:2fact-implies-twotype}
	gives $\ss\ge \czero$.
	
	Analogously, a cotype 2 argument shows $\ss\ge d/2$. 
	We conclude the claim by considering the cases $\cone=0$ and $\cone =d (1/\pp-1/2)$.
\end{proof}

\subsubsection{Proofs for \texorpdfstring{\cref{subsubsection:RKBS}}{}}\label{subsubsection:RKBS_proofs}

\begin{proof}[Proof of \cref{prop:irkbs}]
	Let $\mu\in\calM$ and denote $\abs\mu\equalDef \mu_+ -\mu_- $ for its total variation. 
	For $x\in x$ we define $c_x\equalDef \sup_{y\in X} \abs{\Psi(x,y)\beta(y)}<\infty$
	and observe
	\begin{align}
		\label{eq:irkbs_dual_norm}
		\abs{f_{\mu}(x)}\le \int_X  \abs{\Psi(x,y) \beta(y)} \dx \abs\mu(y) \le c_x \TVnorm{\mu} \le c_x \norm{\mu}_{\calM} \norm{\id:\calM\to \calM(X)} < \infty.
	\end{align}

	Clearly, $E_{\calM,\Psi,\beta}$ is a vector space of functions $X\to \R$ and 
	$\norm{\mycdot}_{E_{\calM,\Psi,\beta}}$ defines a norm on $E_{\calM,\Psi,\beta}$.
	We show that  $(E_{\calM,\Psi,\beta},\norm\mycdot_{E_{\calM,\Psi,\beta}})$ is complete. To this end, define the subspace $\calN\equalDef\{\nu\in \calM\mid f_\nu=0\}\subset \calM$ and observe that $\calN$ is closed in $\calM$ as
	\begin{align*}
		\calN = \bigcap_{x\in X} \ker \phi_x
	\end{align*}
	holds. 
	Therefore, the quotient space $\calM/\calN$ is a Banach space. 
	By construction, the operator $A:\calM/\calN \to E_{\calM,\Psi,\beta}$ given by $[\mu]_{\sim}\mapsto f_{\mu}$ is well-defined and bijective. 
	Furthermore, $A$ is an isometry since for all $\mu\in\calM$ we have
	\begin{align}
		\nonumber
		\norm{f_\mu}_{E_{\calM,\Psi,\beta}} 
		= \inf\{ \norm \nu_{\calM} \mid \nu\in\calM, f_\nu=f_\mu\} = \inf\{ \norm\nu_{\calM}\mid \nu\in\calM, \nu-\mu \in \calN \}  
		= \norm{[\mu]_{\calM/\calN}}_{\calM/\calN}.\nonumber
	\end{align}
	Since $\calM/\calN$ is a Banach space, we hence conclude that $(E_{\calM,\Psi,\beta},\norm\mycdot_{E_{\calM,\Psi,\beta}})$ is a Banach space, too.
	
	Finally, we show that $E_{\calM,\Psi,\beta}$ is a proper BSF: Let $x\in X$ and $f_\mu\in E_{\calM,\Psi,\beta}$. By \eqref{eq:irkbs_dual_norm} 
	we then have 
	\begin{align*}
		\abs{f_\mu(x)}\le c_x  \norm{\id:\calM\to \calM(X)} \cdot   \norm\nu_{\calM}
	\end{align*}
	for all $\nu\in \calM$ with  $f_\nu=f_\mu$. This leads to 
	\begin{align*}
		\abs{f_\mu(x)}\le c_x  \norm{\id:\calM\to \calM(X)} \cdot  \norm{f_\mu}_{E_{\calM,\Psi,\beta}}
	\end{align*}
	by the definition of the norm 
	$\norm{\mycdot}_{E_{\calM,\Psi,\beta}}$, namely $\norm{f_\mu}_{E_{\calM,\Psi,\beta}} = \inf\{\norm \nu_{\calM} \mid \nu\in\calM,f_\nu=f_\mu\}$.
\end{proof}

\begin{proof}[Proof of \cref{lem:difference_of_kernels}]
	Since $\Psi$ is measurable and bounded, \cref{prop:irkbs} shows that  $E_{\calM,\Psi,1}$ is a proper BSF. 
	By \cite[Sect. I.6]{Aronszajn50a}, the kernel $k=k_1+k_2$ indeed corresponds to the RKHS $H_1+H_2$.
	
	For $i\in\{1,2\}$ and $\mu\in\calM$, our assumption  \eqref{eq:requirement_rkhs_decomp} ensures that the map $x\mapsto k_i(\mycdot,x)$ is weakly $\mu$-integrable with Pettis-integral $\int k_i(\mycdot,y) \dx\mu(y)\in H_i$, see \cite[Section 2]{StZi21a} for details.
	For all $x\in X$, this yields
	\begin{align*}
		\bra{\int k_i(\mycdot,y) \dx\mu(y)}(x) = \scal{k_i(\mycdot,x)}{\int k_i(\mycdot,y) \dx\mu(y)}_{H_i} 
		&=\int \scal{k_i(\mycdot,x)}{k_i(\mycdot,y)}_{H_i} \dx\mu(y) \\
		&= \int k_i(x,y) \dx\mu(y) \\
		&= f_{\mu,k_i,1}(x)\, ,
	\end{align*}
	and hence we find $f_{\mu,k_i,1}\in H_i$. We conclude $f_{\mu,\Psi,1} = f_{\mu,k_1,1} -f_{\mu,k_2,1} \in H_1 +H_2 =H$ and $	E_{\calM,\Psi,1} \subset H$. Since both $E_{\calM,\Psi,1}$ and $H$ are proper BSF, \cref{prop:restricted-inclusions} yields the continuous embedding $	E_{\calM,\Psi,1} \hookrightarrow H$.
\end{proof}

\donee

\subsubsection{Proofs for \texorpdfstring{\cref{subsubsection:isotropic_spaces}}{}}\label{def:slob_spaces}

For the definition of Slobodeckij spaces, we essentially follow \cite[2.1.2]{RuSi96},
but we note that our version of the Slobodeckij norm is actually only  norm-equivalent to the norm given there. As discussed previously,
neither our Question \eqref{eq:central_question} nor 2-factorability is affected by switching to equivalent norms.

To begin with, let
  $\Om\subset \Rd$ be a domain with smooth boundary and let $\ss\ge 0, p\ge 1$.
For the coherent set of multi-indices $\A\equalDef\{\al\in\N_0^d\mid \munorm{\al} \le \floor\ss\}$,
let $\preslospace \ss\Om$ be the set of $A$-times weakly differentiable functions defined in \cref{eq:sobolev_auxspace}.
If $\ss\in\N$,  
then we  define for $\fun\in\preslospace{\ss}{\Om}$ the   Sobolev norm just as in \cref{eq:mixed_sob_norm}, that is
\begin{align*}
	\slonorm \fun \ss \pp\Om \equalDef \sup_{\al\in\N_0^d,\munorm{\al} \le \ss} \lpnorm{\partial_\al\fun}{\pp}{\Om} \, .
\end{align*}
Recall that the classical Sobolev norm considers the $p$-sum of the involved terms $\lpnorm{\partial_\al\fun}{\pp}{\Om}$ instead.
Our definition gives, of course, an equivalent norm.

Moreover, if
 $\ss\not\in\N$, we write $\th\equalDef \ss-\floor\ss\in (0,1)$ and define, for $\fun\in\preslospace{\ss}{\Om}$, the Slobodeckij  norm by
\begin{align*}
	\slonorm{\fun}{\ss}{\pp}{\Om}\equalDef \max\Big\{ \sup_{\al\in\N_0^d,\munorm{\al}\le \floor\ss} \lpnorm{\partial_\al \fun}{\pp}{\Om}, \sup_{\al\in\N_0^d ,\munorm{\al}=\floor\ss} \slobodeckij{\partial_\al\fun}{\th}{\pp}{\Om}
	\Big\},
\end{align*}
where  for a measurable function $\gun:\Om\to\R$ the semi-norm $\slobodeckij{g}{\th}{p}{\Om}$ is defined as
\begin{align}
	\label{eq:slobodeckij_core}
	\slobodeckij{g}{\th}{p}{\Om}\equalDef 
	\bra[\Big]{ \int_{\Om}\int_{\Om} \frac{\abs{\gun(x)-\gun(y)}^p}{\norm{x-y}^{\th p + d}} \dx y \dx x }^{1/p}.
\end{align}
In both cases, we define the Slobodeckij space by
\begin{align*}
\slospace \ss \pp \Om \equalDef\set{\fun\in\preslospace{\ss}{\Om} : \slonorm{\fun}{\ss}{\pp}{\Om}<\infty}\, .
\end{align*}

To bound the semi-norm $\slobodeckij{g}{\th}{p}{\Om}$, the following well-known
formula, see e.g.~\cite[Satz 14.8]{Forster17}, which holds for all measurable functions $f:[0,\infty)\to [0,\infty)$ 
and all $0\leq a<b\leq \infty$,
turns out to be useful: 
\begin{align}\label{eq:int-of-radial}
\int_{\Rd} \eins_{[a,b]}(\snorm x) f(\snorm x) \dx x = d V_d \int_a^b f(r) r^{d-1} \dx r\, .
\end{align}
Here $V_d \equalDef \vol(\ball 10)$ is the Volume of the $d$-dimensional unit ball.
For later use we further note that for $z\in \R^d$ and $a \geq 1$, $\be>0$ this formula yields
	\begin{align} \nonumber
	  \int_{\ball {a}{z}} \int_{\Rd\setminus \ball {a+1/2}{z}} \frac{1}{\norm{x-y}^{\b + d}} \dx y \dx x
	&=
	  \int_{\ball {a}{0}} \int_{\Rd\setminus \ball {a+1/2}{0}} \frac{1}{\norm{x-y}^{\b + d}} \dx y \dx x \\ \nonumber
		&\leq
	 \int_{\ball {a}{0}} \int_{\Rd\setminus \ball {a+1/2}{0}} {{\bigl|\snorm x - \snorm y   \bigr| }^{-\b - d}} \dx y \dx x \\ \nonumber
		&\leq
	  \int_{\ball {a}{0}} \int_{\Rd\setminus \ball {a+1/2}{0}}{{\bigl( \snorm y  -a \bigr) }^{-\b - d}} \dx y \dx x \\ \nonumber
	&\leq  a^d  d V_d^2 \int_{a+1/2}^\infty (r-a)^{-\b -d} r^{d-1}   \dx r \\ \label{eq:int-of-radial-apl}
	&<  \frac{2^{\b+d} a^{2d} d V_d^2 }{\b} \, ,
	\end{align}
where in the last step we used $r < 2a (r-a)$ for all $r\geq a+1/2$.

\donee
\begin{lemma}\label{prop:slobo_finite}
    For all $d\geq 1$, $\pp\in [1,\infty)$,  and $\th\in (0,1)$, there exists a constant $c_{d,p,\th}\geq 0$ such that for all 
	open and bounded  $\Om\subset\Rd$ and all $f:\Om\to \R$ that are continuously differentiable with bounded derivative, we have 
	\begin{align*}
	\slobodeckij{ \fun}{\th}{\pp}\Om \leq c_{d,p,\th} |f|_1 (\vol(\Om))^{1/p} (\diam \Om)^{1-\th}\, .
	\end{align*}
\end{lemma}

\begin{proof}
    Since the derivative of $f$ is bounded, $f$ is Lipschitz continuous with  constant $|f|_1<\infty$. 
    We write $\d\equalDef \diam \Om$. 
    For fixed $x\in \Om$ we then have $\snorm{x-y} \leq \d$ for all $y\in \Om$, that is $\Om\subset \ball \d x$.
    This yields
	\begin{align*}
		\slobodeckij \fun {\th} p \Om^p 
		=  
		\int_{\Om}\int_{\Om} \frac{\abs{\fun(x)-\fun(y)}^p}{\norm{x-y}^{{\th} p + d}} \dx y \dx x 
		&\leq 
		|f|_1^p  \int_{\Om}\int_{\Om} \frac{1}{\norm{x-y}^{{(\th - 1)} p + d}} \dx y \dx x  \\
		&\leq 
		|f|_1^p \int_{\Om}\int_{\ball \d x} \frac{1}{\norm{x-y}^{{(\th - 1)} p + d}} \dx y \dx x \\
		&= 
		|f|_1^p \int_{\Om}\int_{\ball \d 0} \frac{1}{\norm{z}^{{(\th - 1)} p + d}} \dx z \dx x \, .
	\end{align*}    
	Moreover, with the help of \eqref{eq:int-of-radial} we obtain
	\begin{align*}
	\int_{\ball \d 0} \frac{1}{\norm{z}^{{(\th - 1)} p + d}} \dx z 
	= 
	d V_d \int_0^\d   r^{-1 + (1-\theta)p} \dx r  
	= 
	d V_d \frac {\d^{(1-\theta)p} }{(1-\theta)p}\, ,
	\end{align*}
	and inserting this into the previous estimate then yields the   assertion since $\Om$ is bounded.
\end{proof}

Analogous to \cref{lemma:asymptotics_in_wsp_of_rescaled_bumps_simple_argument}, we establish the foundation for estimating the type 2 norm and cotype 2 norm of embeddings $\slospace{\ss}{\pone}{\Om} \hookrightarrow \slospace{\tt}{\ptwo}{\Om}$.

\donee

\begin{lemma}\label{lemma:asymptotics_in_wsp_of_rescaled_bumps}
	Let $\Omega\subset\Rd$ be open such that $\ball 10\subset \Om$, and let
	 $\al\in\N_0^d$, $\th\in (0,1)$, and $\pp\in[1,\infty)$.
	 Moreover, for $\d \in (0,1/2]$ and suitable $n\ge 1$  let $x_1,\dots,x_n\in \ball {1/2}0$ be a $3\d$-packing, and
		$\varepsilon_1,\dots,\varepsilon_n \in\set{-1,1}$.
	We define     $\hfun:\Om\to\R$ by
		\begin{align*}
			\hfun(x)\equalDef \sum_{i=1}^n \varepsilon_i\fun(\negpow\d{1} x - \negpow \d 1  x_i)\, ,
		\end{align*}
	where $f$ is the  bump function from \cref{eq:bump_function}. Moreover, we write  $\ss\equalDef \th +\munorm{\al}$.
	Then
	we have $\slobodeckij{\partial_\al\fun}{\th}{p}{\ball 10}>0$ and
	there exists a constant $c_{d,p,\th}>0$
	only depending on $d,p,\th$, such that
		\begin{align*}
			  \slobodeckij{\partial_\al\fun}{\th}{p}{\ball 10}  n^{1/p}\d ^{d/\pp-\ss}  \le \slobodeckij {\partial_\al \hfun}{\th}{\pp}\Om  \le c_{d,p,\th} \bigl(1+\slobodeckij{\partial_\al\fun}{\th}{p}{\ball {3/2}0}\bigr)  n^{1/p}\d ^{d/\pp-\ss} ~.
		\end{align*}
\end{lemma}

\begin{proof}[Proof of \cref{lemma:asymptotics_in_wsp_of_rescaled_bumps}]
    Before we begin with the actual proof, we note 
		  that for any suitable set $\Th\subset \Rd$ and any sufficiently differentiable function $a:\Th\to\R$ 
		one has $\supp(\partial_\al a)\subset \supp(a)$, for any suitable $\al\in \N_0^d$.  This inclusion will be used 
		several times.

		Just as in the proof of \cref{lemma:asymptotics_in_wsp_of_rescaled_bumps_simple_argument} 
		we define $z_i\equalDef \d^{-1}x_i\in \ball{\d^{-1}\!/2}{0}\subset 
		\d^{-1} \Om$. Note that we have  $\norm{z_i-z_j}\ge 3$ for $i\ne j$ as well as the 
		 identity 
		\begin{align*}
			\bra{\partial_\al \hfun}(x)= \negpow\d{\munorm\al } \isum \varepsilon_i\partial_\al\fun\bra{\negpow \d 1x-z_i} \, .
		\end{align*}
		We define $f_i \equalDef \varepsilon_i\partial_\al\fun\bra{\mycdot-z_i}$ and 
		$\gun \equalDef \isum f_i$. For later use we note that $\supp f_i \subset \ball 1{z_i}$, and in particular the functions $f_1,\dots,f_n$ have mutually disjoint support.
		Moreover, by substitution we obtain
		\begin{align} \nonumber
			\slobodeckij{\partial_\al\hfun}{\th}{p}{\Om}^p
			&=\int_{\Om}\int_{\Om} \frac{\abs{\partial_\al\hfun(x)-\partial_\al\hfun(y)}^p}{\norm{x-y}^{\th p + d}} \dx y \dx x\\ \nonumber
			&=\d^{-p \munorm\al}\int_{\Om}\int_{\Om} 
			\frac{
				\abs[\big]{
					\isum \varepsilon_i\partial_\al\fun\bra{\negpow \d 1x-z_i}-\jsum \varepsilon_j\partial_\al\fun\bra{\negpow \d 1y-z_j}}^p
			}
			{\norm{x-y}^{\th  p + d}} 
			\dx y \dx x \\ \nonumber
			&=\d^{2d-p \munorm\al}\int_{\negpow \d 1\Om}\int_{\negpow \d 1\Om} 
			\frac{
				\abs[\big]{
					\gun\bra{x}-\gun\bra{y}}^p
			}
			{\norm{\d x-\d y}^{\th  p + d}} 
			\dx y \dx x \\ \label{lemma:asymptotics_in_wsp_of_rescaled_bumps-h1}
			&=\d^{d-p \ss}\int_{\negpow \d 1\Om}\int_{\negpow \d 1\Om} 
			\frac{
				\abs[\big]{
					\gun\bra{x}-\gun\bra{y}}^p
			}
			{\norm{x-y}^{\th  p + d}} 
			\dx y \dx x \, .
		\end{align}

	Our first goal is to establish the upper bound on $\slobodeckij{\partial_\al\hfun}{\th}{p}{\Om}^p$. To this end, 
	we write $H(x,x) \equalDef 0$ and 
	\begin{align*}
	H(x,y) \equalDef \frac{\abs{\gun\bra{x}-\gun\bra{y}}^p}{\norm{x-y}^{\th  p + d}}\, , \qquad x,y\in \Rd \mbox{ with $x\neq y$.}
	\end{align*}
	Moreover, we write $B_i \equalDef \ball {3/2}{z_i}$. Then, we have $\supp f_i \subset B_i$ and $\lb^d(B_i\cap B_j)= 0$ for all $i\neq j$.
	Finally, we define $B_0 \equalDef \Rd\setminus (B_1\cup\dots\cup B_n)$.
	Since
	we have $H(x,y) = H(y,x)\geq 0 $ for all $x,y\in \Rd$ and $H(x,y) = 0$ for all $x,y\in B_0$
	we then obtain 
	\begin{align*}
	\int_{\Rd}\int_{\Rd} H(x,y) \dx y \dx x
	&= \sum_{i=0}^n \sum_{j=0}^n \int_{B_i}\int_{B_j} H(x,y) \dx y \dx x \\
	&= \sum_{i=1}^n \int_{B_i}\int_{B_i} H(x,y) \dx y \dx x 
	+ \sum_{i=0}^n \sum_{j\neq i} \int_{B_i}\int_{B_j} H(x,y) \dx y \dx x  \\
	&= n \slobodeckij{\partial_\a \fun}{\th}{\pp}{\ball {3/2}{0}}  + \sum_{i=0}^n \sum_{j\neq i} \int_{B_i}\int_{B_j} H(x,y) \dx y \dx x\, ,
	\end{align*}
	where in the last step we used 
		\begin{align*}
		\int_{B_i}\int_{B_i} H(x,y) \dx y \dx x 
		&= 
		\int_{\ball {3/2}{z_i}}\int_{\ball {3/2}{z_i}} \frac{\abs{\partial_\al\fun\bra{x-z_i} -\partial_\al\fun\bra{y-z_i}}^p}{\norm{x-y}^{{\th} p + d}} \dx y \dx x \\ 
        &= 
		\int_{\ball {3/2}{0}}\int_{\ball {3/2}{0}} \frac{\abs{\partial_\al\fun\bra{x} -\partial_\al\fun\bra{y}}^p}{\norm{x-y}^{{\th} p + d}} \dx y \dx x \\
		&= \slobodeckij{\partial_\a \fun}{\th}{\pp}{\ball {3/2}{0}}^p \, .
		\end{align*}
	To treat the double sum, we first observe 
	\begin{align*}
	\sum_{i=0}^n \sum_{j\neq i} \int_{B_i}\int_{B_j} H(x,y) \dx y \dx x
	&=  \sum_{j=1}^n \int_{B_0} \int_{B_j} H(x,y) \dx y \dx x + \sum_{i=1}^n \sum_{j\neq i}\int_{B_i} \int_{B_j} H(x,y) \dx y \dx x \\
	&\leq \sum_{j=1}^n \int_{\Rd\setminus B_j} \int_{B_j} H(x,y) \dx y \dx x + \sum_{i=1}^n  \int_{B_i} \int_{\Rd\setminus B_i} H(x,y) \dx y \dx x \\
	&= 2 \sum_{i=1}^n  \int_{B_i} \int_{\Rd\setminus B_i} H(x,y) \dx y \dx x\, .
	\end{align*}
	In addition, we have 
	\begin{align*}
	 \int_{B_i} \int_{\Rd\setminus B_i} H(x,y) \dx y \dx x  
	& = 
	 \int_{\ball 1{z_i}} \int_{\Rd\setminus B_i} H(x,y) \dx y \dx x \\
	&\qquad  +  \int_{B_i\setminus \ball 1{z_i}} \int_{\Rd\setminus \ball {2}{z_i}} H(x,y) \dx y \dx x \\
	&\qquad   +  \int_{B_i\setminus \ball 1{z_i}} \int_{\ball {2}{z_i}\setminus B_i} H(x,y) \dx y \dx x \, .
	\end{align*}
	In the following, we estimate these three double  integrals.  
	The first one can be estimated by 
	\begin{align*}
	\int_{\ball 1{z_i}} \int_{\Rd\setminus B_i} H(x,y) \dx y \dx x 
	&= 
	\int_{\ball 1{z_i}} \int_{\Rd\setminus \ball {3/2}{z_i}} \frac{\abs{\gun\bra{x}-\gun\bra{y}}^p}{\norm{x-y}^{\th  p + d}} \dx y \dx x \\
	&\leq 
	\int_{\ball 1{z_i}} \int_{\Rd\setminus \ball {3/2}{z_i}} \frac{2^p\inorm g^p}{\norm{x-y}^{\th  p + d}} \dx y \dx x \\
	&\leq  
	\frac{2^{p(1+\th)+d}  d V_d^2 }{\th p} \, ,
	\end{align*}
	where we used $\inorm g = \inorm f = 1$ and \eqref{eq:int-of-radial-apl}. %
	The second double integral  can be analogously estimated by 
	\begin{align*}
	 \int_{B_i\setminus \ball 1{z_i}} \int_{\Rd\setminus \ball {2}{z_i}} H(x,y) \dx y \dx x
	&= 
	\int_{B_i\setminus \ball 1{z_i}} \int_{\Rd\setminus \ball {2}{z_i}}\frac{\abs{\gun\bra{x}-\gun\bra{y}}^p}{\norm{x-y}^{\th  p + d}} \dx y \dx x \\
	&=
	\int_{B_i\setminus \ball 1{z_i}} \int_{\Rd\setminus \ball {2}{z_i}}\frac{\abs{\gun\bra{y}}^p}{\norm{x-y}^{\th  p + d}}\dx y \dx x \\
	&\leq
	  \int_{\ball {3/2}{z_i}} \int_{\Rd\setminus \ball {2}{z_i}} \frac{1}{\norm{x-y}^{\th  p + d}} \dx y \dx x \\
	&\leq
	\frac{2^{\th p+ 3d}  d V_d^2 }{\th p} \, .
	\end{align*}	
	For the third double integral we observe that for $x\in B_i\setminus \ball 1{z_i}$ we have $g(x) = 0$ and for $y\in \ball {2}{z_i}\setminus B_i$ we also have $g(y)= 0$. This shows 
	\begin{align*}
	\int_{B_i\setminus \ball 1{z_i}} \int_{\ball {2}{z_i}\setminus B_i} H(x,y) \dx y \dx x = 0\,.
	\end{align*}
	Combining these considerations, we find  
	\begin{align*}
	\slobodeckij{\partial_\al\hfun}{\th}{p}{\Om}^p 
	\leq \d^{d- ps} n \slobodeckij{\partial_\a \fun}{\th}{\pp}{\ball {3/2}{0}}^p
	+2  \d^{d- ps} n \Bigl( \frac{2^{p(1+\th)+d}  d V_d^2 }{\th p} + \frac{2^{\th p+3d} d V_d^2 }{\th p} \Bigr)    \, .
	\end{align*}

	Finally we establish the lower bound for the Slobodeckij norm of $\hfun$ using the fact that $f_1,\dots,f_n$ have mutually disjoint support: 
		\begin{align*}
			\int_{\negpow \d 1\Om}\int_{\negpow \d 1\Om}
			\frac{\abs{\gun(x)-\gun(y)}^p	}
			{\norm{x-y}^{\th  p + d}} 
			\dx y \dx x
			\ge&\, \isum \int_{\ball 1{z_i}} \int_{\ball 1{z_i}} 
			\frac{\abs{\gun(x)-\gun(y)}^p	}
			{\norm{x-y}^{\th  p + d}} 
			\dx y \dx x\\
			=&\, \isum \int_{\ball 1{z_i}} \int_{\ball 1{z_i}} \frac{
				\abs{
					\varepsilon_i\partial_\al\fun\bra{x-z_i}- \varepsilon_i\partial_\al\fun\bra{y-z_i}}^p
			}
			{\norm{x-y}^{\th  p + d}} 
			\dx y \dx x\\
			=&\,n \int_{\ball 10} \int_{\ball 10} \frac{
			\abs{
			\partial_\al\fun\bra{x}- \partial_\al\fun\bra{y}}^p
			}
			{\norm{x-y}^{\th  p + d}} 
			\dx y \dx x
			\\
			=&\, n\slobodeckij{\partial_\al\fun}{\th}{p}{\ball 10}^p.
		\end{align*}
		Moreover,
		since $\partial_\al\fun$ is continuous and not constant, we have $\slobodeckij{\partial_\al\fun}{\th}{p}{\ball 10}>0$.
\end{proof}

Now, we are able to estimate the Rademacher sequence norm and the $\ell_2$-sequence norm for bump functions in Slobodeckij spaces.

\begin{lemma}\label{lemma:slobo_rad_and_elltwo}
	Let $\Omega\subset\Rd$ be open such that $\ball 10\subset \Om$ and let $\ss\ge 0$, $\pp\in [1,\infty)$.
		For $\d\in(0,1/2]$ and suitable  $n\geq 1$, let $x_1,\dots,x_n\in \ball{1/2}0$ be a $3\d$-packing and define $\fun_i:\Om\to\R$ by
		\begin{align*}
			\fun_i(x)\equalDef \fun(\d^{-1} x -\d^{-1} x_i),
		\end{align*}
		where $\fun$ is the bump function considered in \cref{eq:bump_function}. 
		Then,
		there exist   constants $c>0$ and $C>0$ that are independent of $\d$ and $n$ such that we have
		\begin{align*}
			c
			n^{1/\pp} \d^{d/\pp -\ss} 
			\le&\, \radnorm{(\fun_1,\dots,\fun_n)}n{\slospace \ss\pp\Om} 
			\le C
			n^{1/\pp} \d^{d/\pp -\ss} 
			\intertext{and}
			c
			n^{1/2} \d^{d/\pp -\ss} \le&\, \elltwonorm{(\fun_1,\dots,\fun_n)}n{\slospace \ss\pp\Om}\le C n^{1/2} \d^{d/\pp -\ss} \, .
		\end{align*}
	\end{lemma}

\begin{proof}
	In the case $s\in \N$, the assertion has already be shown in Lemma \ref{lem:radnorm-mixed-sobol}. In the following we thus assume
	$s\not\in \N$ and define $\th\equalDef \ss-\floor\ss$. Finally,
	we fix some $\al_0\in\N_0^d$ with $\munorm{\al_0} =\floor \ss$.

	Now \cref{lemma:asymptotics_in_wsp_of_rescaled_bumps} directly yields the lower bound on the Rademacher norm
	\begin{align*}
		\slobodeckij{\partial_{\al_0}\fun}{\th}{p}{\ball 10}  n^{1/p}\d ^{d/\pp-\ss}   \le \E_{\radsequence{\varepsilon}} \slonorm{\varepsilon_1 \fun_1 +\dots +\varepsilon_n \fun_n}{\ss}{\pp}{\Om} =  \radnorm{(\fun_1,\dots,\fun_n)}n{\slospace \ss\pp\Om}.
	\end{align*}
	For the corresponding upper bound we write
	$C_{\a,d,p,\th}\equalDef  c_{d,p,\th}  (1+\slobodeckij{\partial_\al\fun}{\th}{p}{\ball {3/2}0})$
	From \cref{lemma:asymptotics_in_wsp_of_rescaled_bumps_simple_argument,lemma:asymptotics_in_wsp_of_rescaled_bumps}
	we then obtain
	\begin{align*}
		\radnorm{(\fun_1,\dots,\fun_n)}{n}{\slospace \ss\pp\Om}
		&= \E_{\radsequence{\varepsilon}} \slonorm{\varepsilon_1 \fun_1+\dots +\varepsilon_n \fun_n}{\ss}{\pp}{\Om}\\
		&\le\,\max\{\sup_{\munorm\al\le \floor\ss} \lpnorm{\partial_\al \fun}{\pp}{\Om} n^{1/\pp}\d^{d/\pp -\munorm\al}
		,
		\sup_{\munorm\al=\floor\ss} C_{\a,d,p,\th} n^{1/\pp} \d^{d/\pp -\ss}  \}.
	\end{align*}
	Since $\d\le 1/2 $, we see that the asserted upper bound for the Rademacher norm holds for sufficiently large $C>0$.
	
	Moreover, applying \cref{lemma:asymptotics_in_wsp_of_rescaled_bumps_simple_argument,lemma:asymptotics_in_wsp_of_rescaled_bumps} in the case $n=1$ shows for all $i=1,\dots , n$ the estimate
	\begin{align*}
		\slobodeckij{\partial_{\al_0}\fun}{\th}{p}{\ball 10}     \d^{d/\pp -\ss} \le
		\, \slonorm{\fun_i }{\ss}{\pp}{\Om}
		&\le
		\max\{\sup_{\munorm\al\le \floor\ss} \lpnorm{\partial_\al \fun}{\pp}{\Om} \d^{d/\pp -\munorm\al}
		,
		\sup_{\munorm\al=\floor\ss}\d^{d/\pp -\ss} C_{\a,d,p,\th}\}\\
		&\leq C  \d^{d/\pp -\ss} \, .
	\end{align*}
	From this we easily obtain the bounds on the $\ell_2$-sequence norms.
\end{proof}

The following theorem relates different spaces of Besov-Triebel-Lizorkin type to each other and is at the heart of the constructive statement in \cref{lemma:slobotwofac}.

\donee
\begin{theorem}\label{theorem:triebels_embeddings}
	Let $\Om\subset \Rd$ be a bounded domain with smooth boundary and
	let $s,t\geq 0$, $p,p_1,p_2\in [1,\infty]$, and $q,q_1,q_2\in [1,\infty]$.
	Then the following statements hold true:
	\begin{enumerate}
		\item 
		\textbf{Lowering the smoothness   to increase the integration index.}
		If $\ss>\tt$, then we have
		\begin{alignat*}{2}
			\triebelspace{\ss}{\pone}{\qone} \Om &\hookrightarrow\triebelspace{\tt}{\ptwo}{\qtwo} \Om\, , \qquad&& \text{ if }\ss-\tt \ge \frac{d}{\pone} -\frac{d}{\ptwo} \mbox{ and } p_1,p_2<\infty,\\
			\besovspace{\ss}{\pone}{\qone} \Om &\hookrightarrow \besovspace{\tt}{\ptwo}{\qtwo} \Om\, , && \text{ if }\ss-\tt > \frac{d}{\pone} -\frac{d}{\ptwo}.
		\end{alignat*}
		\item 
		\textbf{Lowering the integration index.}
		If $\pone\ge \ptwo$, then we have
		\begin{align*}
			\triebelspace{\ss}{\pone}{q}\Om &\hookrightarrow\triebelspace{\ss}{\ptwo}{q}\Om\, , \qquad& \text{ if } p_1,p_2,q<\infty, \\
			\besovspace{\ss}{\pone}{q}\Om&\hookrightarrow \besovspace{\ss}{\ptwo}{q}\Om \, .
		\end{align*}
		\item 
		\textbf{Changing the fein index.} If $\qone\le\qtwo$, then
		for all $\e>0$ we have
		\begin{align*}
			\triebelspace{\ss}{p}{\qone}\Om &\hookrightarrow\triebelspace{\ss}{p}{\qtwo}\Om\, , \qquad& \text{ if } p<\infty, \\
			\triebelspace{\ss+\varepsilon}{p}{\infty}\Om &\hookrightarrow\triebelspace{\ss}{p}{q}\Om\, , \qquad& \text{ if } p<\infty, \\
			\besovspace{\ss}{p}{\qone}\Om &\hookrightarrow\besovspace{\ss}{p}{\qtwo}\Om,\\
			\besovspace{\ss+\varepsilon}{p}{\infty}\Om &\hookrightarrow\besovspace{\ss}{p}{q}\Om.
		\end{align*}
		\item 
		\textbf{Changing between Besov and Triebel spaces.}
		If $\ss>\tt$ and $\ss-\tt> d(1/\pone -1/\ptwo)$, then we have
		\begin{align*}
			\triebelspace \ss\pone \qone \Om&\hookrightarrow \besovspace \tt\ptwo\qtwo\Om, \qquad&\text{ if } \pone <\infty,\\
			\besovspace \ss\pone \qone \Om&\hookrightarrow \triebelspace \tt\ptwo\qtwo\Om\qquad&\text{ if } \ptwo <\infty.
		\end{align*} 
	\end{enumerate}
\end{theorem}

\begin{proof}
	\ada i See e.g.~\cite[Parts (i) and (ii) of the theorem on p.~196/7]{Triebel83}.
	
	\ada  {ii} See  e.g.~\cite[Part (iii) of the theorem  on p.~196/7]{Triebel83}.
	
	\ada {iii} In the case  $\Om = \Rd$, these embeddings can be found in e.g.~\cite[Prop.~2 on p.~47]{Triebel83}.
	The restriction to bounded domains is discussed in e.g.~\cite[Sec.~2.4.4]{RuSi96}.

	\ada {iv}  We only show $\triebelspace \ss\pone \qone \Om\hookrightarrow \besovspace \tt\ptwo\qtwo\Om$, the other embedding can be proven analogously. 

	Recall the identity $\triebelspace \ss\pp\pp\Om =\besovspace \ss\pp\pp\Om$ for $\pp<\infty$, see \cref{eq:triebel-besov-identity}. 
	We set $\varepsilon \equalDef \ss-\tt - d(1/\pone -1/\ptwo)>0$.
	By \emph{i)} we obtain $\triebelspace \ss\pone\qone\Om \hookrightarrow \triebelspace {\ss-\varepsilon/2}\pone \pone \Om =\besovspace {\ss-\varepsilon/2}\pone\pone \Om$
	as well as $\besovspace {\ss-\varepsilon/2}\pone\pone \Om \hookrightarrow \besovspace \tt\ptwo\qtwo\Om$.
\end{proof}

\donee
\begin{proof}[Proof of \cref{lemma:slobotwofac}]
	We first note that Part \emph{i)} of \cref{theorem:triebels_embeddings}
	together with \cref{eq:triebelid-slobo} gives
	\begin{align}\label{lemma:slobotwofac-h1}
	\slospace{\ss}{\pone}{\Om}
	= \triebelspace{\ss}{\pone}{\qone}{\Om}
	\hookrightarrow\triebelspace{\tt}{\ptwo}{\qtwo}{\Om}
	= \slospace{\tt}{\ptwo}{\Om}\, ,
	\end{align}
	where $\qone$ and $\qtwo$ are  defined according to \cref{eq:triebelid-slobo}.

	\ada i
	Let us fix a $\uu\in (\tt+\ctwo,\ss-\cone)$. Then Part \emph{i)} of \cref{theorem:triebels_embeddings} shows
	\begin{align*}
	\triebelspace{\ss}{\pone}{\qone} \Om
	\hookrightarrow
	\triebelspace{u}{2}{2} \Om
	\hookrightarrow
	\triebelspace{\tt}{\ptwo}{\qtwo}\Om \, .
	\end{align*}
	Combining this with \eqref{lemma:slobotwofac-h1} and \eqref{eq:triebelid-sob} yields the assertion. 
	
	\ada {ii}
	By a translation and scaling argument, we can assume $\ball 10\subset \Om$ without loss of generality.
	Let us fix a
	  $0<\d\le 1/2$.  We define $n\equalDef\paceuclid{\d}{\ball{1/2}{0}}$ and
	  choose $\fun_1,\dots,\fun_n$ as in \cref{lemma:slobo_rad_and_elltwo}.
	  Then \cref{lemma:slobo_rad_and_elltwo} gives
	\begin{align*}
		 \typenorm{\slospace \ss\pone\Om\hookrightarrow\slospace\tt\ptwo\Om} 
		 \ge \frac{ \radnorm{(\fun_1,\dots,\fun_n)}{n}{\slospace \ss\pone \Om} }{\elltwonorm{(\fun_1,\dots,\fun_n)}{
		 	n}{\slospace \tt\ptwo \Om}}
	 	\ge \ha C n^{1/\ptwo -1/2} \d^{\ss-\tt -d(1/\pone -1/\ptwo))}
	\end{align*}
	for constant $\ha C>0$ that is independent of $\d$ and $n$.
	Using the packing number bound \eqref{eq:cov-num-Rd} we thus find a constant $\bar C>0$ such that
	\begin{align*}
		\typenorm{\slospace \ss\pone\Om\hookrightarrow\slospace\tt\ptwo\Om} 
		\ge
		\bar C \d^{\ss-\tt +d(1/2 -1/\pone)}
	\end{align*}
	holds for all $\d\in(0,1/2]$.
	Now, \cref{lem:2fact-implies-twotype} in combination with $s>t$ implies $\ss-\tt\ge \cone$.
	Using a cotype 2 argument, we analogously obtain the requirement $\ss-\tt\ge \ctwo$.
	
	It remains to show that $\ss-\tt\ge \cone+\ctwo$ holds. If $\cone=0$ or $\ctwo=0$, this follows directly from our previous considerations. 
	Moreover, if both expressions are positive, we have
	\begin{align*}
		\cone +\ctwo = d(1/\pone -1/\ptwo),
	\end{align*}
	and the claim holds by assumption.
\end{proof}

		The following lemma shows that
		Slobodeckij spaces are \emph{dense }in the family of
		Besov-Triebel-Lizorkin spaces in a suitable way.

\begin{restatable}{lemma}{epsapproxwithslobospaces}\label{lemma:epsapprox}
	Let $\Om\subset\Rd$ be a bounded  domain with smooth boundary, $s> 0$, and let
	$\spaceone{\ss}{\pp}{\qq}{\Om}$ be a Besov-Triebel-Lizorkin space.
	Then, for any $\varepsilon_1,\varepsilon_2>0$ such that $\min\{1,\ss\}>\varepsilon_1>\varepsilon_2$ 
	there exist non-integer smoothness  
	parameters $0< \ch\ss < \ss <\ha\ss$  and integration parameters $\ha\pp,\ch\pp \in [1,\infty)$
	fulfilling
	\begin{align}
		\label{eq:epsapprox_spaces}
		\slospace{\ha\ss}{\ha\pp}{\Om} \hookrightarrow \spaceone{\ss}{\pp}{\qq}{\Om}\hookrightarrow \slospace{\ch\ss}{\ch\pp}{\Om},
	\end{align}
	and %
	\begin{align}
		\label{eq:epsapproximation}
		\varepsilon_1\ge  \max\{\ha  \ss -\ss,\ss-\ch\ss \}\ge \min\{\ha  \ss -\ss,\ss-\ch\ss\}\ge \varepsilon_2\ge  \max\bigl\{\abs{d/\pp-d/\ha\pp},\abs{d/\pp-d/\ch\pp}\bigr\} \, .
	\end{align}
\end{restatable}

\begin{proof}
	We choose 
	$\ha\ss \in (\ss+\varepsilon_2,\ss+\varepsilon_1)\setminus \N$ 
	and 
	$\ch\ss\in (\ss-\varepsilon_1,\ss-\varepsilon_2)\setminus \N$ which gives $\ha \ss > s > \ch \ss$.
	
	In the case  $\pp<\infty$ we define 
	$\ha \pp=\ch\pp=\pp$. 
	Then, \eqref{eq:epsapproximation} is fulfilled and \eqref{eq:triebelid-slobo}
	provides the identities
	$\slospace{\ha \ss}{\ha\pp}{\Om}=\triebelspace{\ha\ss}{\ha\pp}{\ha\pp}{\Om}$ and $\slospace{\ch\ss}{\ch\pp}{\Om}=\triebelspace{\ch\ss}{\ch\pp}{\ch\pp}{\Om}$. 
	Now, in the Triebel-Lizorkin case, part  \emph{i)}  of \cref{theorem:triebels_embeddings} yields
	\begin{align*}
		\triebelspace{\ha\ss}{\ha\pp}{\ha\pp}{\Om}
\hookrightarrow 
\triebelspace{\ss}{\pp}{q}{\Om} 
\hookrightarrow 
\triebelspace{\ch\ss}{\ch\pp}{\ch\pp}{\Om} \, .
	\end{align*}
	Analogously, in the Besov case, part \emph{iv)}  of \cref{theorem:triebels_embeddings} gives
	$\triebelspace{\ha\ss}{\ha\pp}{\ha\pp}{\Om} 
\hookrightarrow 
\besovspace{\ss}{\pp}{q}{\Om} 
\hookrightarrow 
\triebelspace{\ch\ss}{\ch\pp}{\ch\pp}{\Om}$.

In the case $p=\infty$ we have $\spaceone\ss\pp\qq\Om=\besovspace\ss\infty\qq\Om$.
Define $\ha\pp=\ch\pp\equalDef 2d/\varepsilon_2 \in [2,\infty)$. Then, \eqref{eq:epsapproximation} is fulfilled and \eqref{eq:epsapprox_spaces} follows 
from part \emph{iv)}  of \cref{theorem:triebels_embeddings}, namely
$\triebelspace{\ha\ss}{\ha\pp}{\ha\pp}{\Om} 
\hookrightarrow 
\besovspace{\ss}{\infty}{q}{\Om} 
\hookrightarrow 
\triebelspace{\ch\ss}{\ch\pp}{\ch\pp}{\Om}$.
\end{proof}

\begin{proof}[Proof of \cref{thm:two_factorisability_of_general_triebel_besov_embedds}]
	By \cref{theorem:triebels_embeddings} \emph{i)} and \emph{iv)}, the assumption $\ss-\tt >d(1/\pone -1/\ptwo)$ or, if both spaces are of Triebel-Lizorkin type the assumption $\ss-\tt\ge d(1/\pone -1/\ptwo)$, ensures that the embedding $\embe{\spaceone\ss\pone\qone\Om }{\spacetwo \tt\ptwo \qtwo \Om}$ exists.

	\ada i %
	The existence of embeddings $\embe{\spaceone \ss\pone\qone \Om }{\sobspace \uu 2\Om$} and $\embe{\sobspace\uu 2\Om}{\spacetwo\tt\ptwo\qtwo\Om}$ constituting the 2-factorization follows directly from \cref{theorem:triebels_embeddings} \emph{i)} and \emph{iv)}, using the identity $\sobspace{\uu}{2}{\Om}= \triebelspace \uu 22\Om$, see \cref{eq:triebelid-sob}.

	\ada {ii}
	Let the embedding $\embe{\spaceone \ss\pone \qone \Om}{\spacetwo \tt\ptwo \qtwo\Om}$ be 2-factorable as
	\begin{center}
		\begin{tikzcd}
			\spaceone\ss\pone\qone\Om  \arrow[rd,"U"] \arrow[hookrightarrow,rr, "\id"]&  & \spacetwo \tt\ptwo \qtwo \Om 
			\\
			 & H \arrow[ru,"V"] & 
		\end{tikzcd},
	\end{center}
	where $H$ is a Hilbert space and $U$ and $V$ are bounded linear operators.
	
Choose $\tt>\varepsilon>0$.
We apply \cref{lemma:epsapprox},
where we set $\varepsilon_1 \equalDef \varepsilon$ and $\varepsilon_2\equalDef \varepsilon/2$,
to find Slobodeckij spaces 
$\slospace{\ha \ss}{\ha \pone}\Om \hookrightarrow\spaceone \ss\pone\qone\Om\hookrightarrow \slospace {\ch \ss} {\ch \pone }\Om$ 
and 
$\slospace{\ha \tt}{\ha \ptwo}\Om \hookrightarrow \spacetwo\tt\ptwo\qtwo\Om\hookrightarrow \slospace {\ch \tt} {\ch \ptwo }\Om$ with
non-integer smoothness parameters $0<\ch \tt\le \tt <\ss \le\ha \ss$ and integration parameters $1\le \ha\pone,\ch\ptwo$ fulfilling
\begin{align}
	\label{eq:approximative_params}
	\varepsilon \ge \max\{\ha\ss-\ss, \tt-\ch\tt\}\ge \min\{\ha\ss-\ss, \tt-\ch\tt\} \ge \max\{\abs{d/\pone -d/\ha\pone}, \abs{d/\ptwo -d/\ch\ptwo} \}.
\end{align}
We obtain the embedding $\slospace{\ha \ss}{\ha \pone}\Om \hookrightarrow \slospace {\ch \tt} {\ch \ptwo }\Om $, 
which can be 2-factorized as
\begin{center}
	\begin{tikzcd}
		\slospace {\ha\ss}{\ha\pone}\Om\arrow[hookrightarrow,r, "\id"] \arrow[rightarrow,rrd, 
		"U\circ \id", swap
		]&
		\spaceone\ss\pone\qone\Om  \arrow[rd,"U"] \arrow[hookrightarrow,rr, "\id"]&  & \spacetwo \tt\ptwo \qtwo \Om \arrow[hookrightarrow,r, "\id"] 
		&\slospace{\ch\tt}{\ch\ptwo}\Om
		\\
		& & H \arrow[ru,"V"]\arrow[rru,
		"\id\circ V", swap
		] & &
	\end{tikzcd}.
\end{center}
Now, we estimate
\begin{align*}
	\ha\ss-\ch \tt -d(1/\ha\pone -1/\ch\ptwo) &= \ha\ss -\ss +\tt-\ch\tt + d/\pone -d/\ha\pone +d/\ch\ptwo - d/\ptwo  +\ss-\tt- d(1/\pone-1/\ptwo) \\
	&\ge \ha\ss -\ss +\tt-\ch\tt  - \abs{d/\pone -d/\ha\pone}-  \abs{d/\ptwo -d/\ch\ptwo} +\ss-\tt- d(1/\pone-1/\ptwo) \\
	&\ge \ss-\tt- d(1/\pone-1/\ptwo),
\end{align*}
using \eqref{eq:approximative_params} in the last step. The assumption $\ss-\tt\ge d(1/\pone -1/\ptwo)$ yields
\begin{align*}
	\ha \ss-\ch\tt \ge&\, d(1/\ha\pone -1/\ch\ptwo) \, ,
\end{align*}
and hence  \cref{lemma:slobotwofac}  shows $\ha\ss-\ch\tt \ge \pospart{d/\ha\pone-d/2}+\pospart{d/2-d/\ch\ptwo}$.
We now leverage from the general estimate $(a+b)_+ \leq a_+ + |b|$, which holds for all $a,b\in \R$, and the
inequalities in
\eqref{eq:approximative_params} to estimate
\begin{align*}
	&\ss-\tt -\cone-\ctwo\\
	 \ge&\, 
	 \ha\ss-\ch\tt -\pospart{d/\ha\pone -d/2} -\pospart{ d/2 -d/\ch\ptwo} - \abs{\ss-\ha\ss }-\abs{\ch\tt-\tt} -\abs{d/\pone-d/\ha\pone} -\abs{d/\ptwo -d/\ch\ptwo}  \\
	 \ge&\,\ha\ss-\ch\tt -\pospart{d/\ha\pone -d/2} -\pospart{ d/2 -d/\ch\ptwo} -4\varepsilon\\
	 \ge&\, -4\varepsilon \, .
\end{align*}
Since $\varepsilon>0$ is arbitrarily small, the claim follows.
\end{proof}

\begin{proof}[Proof of \cref{lemma:besov_c0_twofac}]
	We choose $\ha\ss\in (0,\ss-d/\pp)$. Then \eqref{eq:triebelid-holder} states $\besovspace {\ch\ss}\infty\infty\Om=\holspace {\ch\ss}\Om$ and we clearly have $\holspace {\ch\ss}\Om\hookrightarrow C^0(\Om)$.
	Parts \emph{i)} and \emph{iv)} of \cref{theorem:triebels_embeddings} yield the existence of the embedding
	\begin{align*}
		\spaceone\ss \pp\qq\Om \hookrightarrow \besovspace{\ha \ss}{\infty}{\infty}{\Om} \hookrightarrow C^0(\Om).
	\end{align*}

	\ada i 
	Let $\uu\in(d/2,\ss-\czero)$ and choose $\ha\ss\equalDef (\uu-d/2)/2>0$. 
	Then, by $\sobspace \uu  2\Om= \triebelspace{u}{2}{2}{\Om}$, see \eqref{eq:triebelid-sob},
	and 
	\cref{thm:two_factorisability_of_general_triebel_besov_embedds} we have
	\begin{align*}
		\spaceone\ss\pp\qq\Om\hookrightarrow \sobspace \uu 2\Om \hookrightarrow \besovspace{\ha\ss}{\infty}{\infty}{\Om}\hookrightarrow C^0(\Om)
	\end{align*}
	and the claim follows.

	\ada {ii}
	Let the embedding $\embe{\spaceone \ss\pp \qq \Om}{C^0(\Om)}$ be 2-factorable as
	\begin{center}
		\begin{tikzcd}
			\spaceone\ss\pp\qq\Om  \arrow[rd,"U"] \arrow[hookrightarrow,rr, "\id"]&  & C^0(\Om)
			\\
			& H \arrow[ru,"V"] & 
		\end{tikzcd},
	\end{center}
	where $H$ is a Hilbert space and $U$ and $V$ are bounded linear operators.
	Fix some $\varepsilon>0$. By \cref{lemma:epsapprox} we find
	non-integer $\ha\ss >\ss$ and some $ \ha\pp\ge 1 $ such that 
	\begin{align}
		\label{eq:another_param_ineq}
		\varepsilon > \ha\ss-\ss \ge \abs{d/\pp-d/\ha\pp}
	\end{align} holds and such that the embedding
$
		 \slospace {\ha\ss} {\ha\pp} \Om\hookrightarrow \spaceone \ss\pp\qq\Om 
$
	exists. 	
	Especially, the embedding $\slospace{\ha\ss}{\ha\pp}\Om\hookrightarrow C^0(\Om)$ is 2-factorable as
	\begin{center}
		\begin{tikzcd}
			\slospace {\ha\ss} {\ha\pp} \Om\arrow[hookrightarrow, r,"\id"] \arrow[rrd,swap,"U\circ\id"] &\spaceone\ss\pp\qq\Om  \arrow[rd,"U"] \arrow[hookrightarrow,rr, "\id"] &  & C^0(\Om)
			\\
			& & H \arrow[ru,"V"] & 
		\end{tikzcd}.
	\end{center}
    	By a translation and scaling argument, we can now assume $\ball 10\subset \Om$ without loss of generality.
	
	Let us fix a
	  $0<\d\le 1/2$.  We define $n\equalDef\paceuclid{\d}{\ball{1/2}{0}}$ and
	  choose $\fun_1,\dots,\fun_n$ as in \cref{lemma:slobo_rad_and_elltwo}. Since those functions are disjointedly supported translated copies of each other, we have
	  \begin{align*}
	  	\radnorm{(\fun_1,\dots,\fun_n)}{n}{C^0(\Om)}=&\,\E_{\radsequence{\varepsilon}} \norm{ \varepsilon_1 \fun_1 +\dots+\varepsilon_n\fun_n}_{C^0(\Om)}= \supnorm{\fun_1}=1.
	  \end{align*}
	  This observation and \cref{lemma:slobo_rad_and_elltwo}  give
	\begin{align*}
		 \typenorm{\embe{\slospace {\ha\ss} {\ha\pp} \Om }{C^0(\Om)}}
		 \ge 
		 \frac{ \radnorm{(\fun_1,\dots,\fun_n)}{n}{C^0(\Om)}} {\elltwonorm{(\fun_1,\dots,\fun_n)}{
		 	n}{\slospace {\ha\ss} {\ha\pp} \Om }}
	 	\ge 
	 	\ha C n^{-1/2} \d^{\ha \ss -d/\hat \pp }
	\end{align*}
	for constant $\ha C>0$ that is independent of $\d$ and $n$.
	Using the packing number bound \eqref{eq:cov-num-Rd} we thus find a constant $\bar C>0$ such that
	\begin{align*}
		\typenorm{\embe{\slospace {\ha\ss} {\ha\pp} \Om }{C^0(\Om)}}
		\ge
		\bar C \d^{\ha\ss -d(1/\ha\pp-1/2)}
	\end{align*}
	holds for all $\d\in(0,1/2]$.
	Now, \cref{lem:2fact-implies-twotype} implies $\ha\ss\ge \pospart{d/\ha\pp-d/2}$.
	Using a cotype 2 argument, we analogously obtain the requirement $\ha\ss\ge d/2$, where we use $	\elltwonorm{(\fun_1,\dots,\fun_n)}{n}{C^0(\Om)}=n^{1/2}$.
	
	Since $\varepsilon$ is arbitrarily small, we conclude
	$\ss\ge \czero$ and $\ss\ge d/2$ from \cref{eq:another_param_ineq}.
	It remains to show that $\ss\ge \cone+d/2$ holds. If $\czero=0$, this follows directly from our previous considerations. 
	If $\czero>0$, then we have
	\begin{align*}
		\czero +d/2 = d/\pp,
	\end{align*}
	and the claim holds by assumption.
\end{proof}

\paragraph*{Acknowledgements}
This work is funded by the Deutsche Forschungsgemeinschaft (DFG) in the project STE 1074/5-1, within the DFG priority programme SPP 2298 “Theoretical Foundations of Deep Learning”.
We would like to thank Jens Wirth and David Holzm\"uller for valuable discussions and insights that contributed to this work.

\bibliographystyle{plain}
\bibliography{steinwart-mine,steinwart-books,steinwart-article,steinwart-preprint,steinwart-proc}

\end{document}